\documentclass[11pt]{amsart}

\usepackage{amsmath,amssymb,enumerate}

\usepackage{epsfig,fancyhdr,color}%,showkeys,amsmidx

\usepackage{amssymb}
\usepackage{amsmath,amsthm}
\usepackage{latexsym}
\usepackage{amscd}
\usepackage{psfrag}
\usepackage{graphicx}
\usepackage[latin1]{inputenc}
\usepackage[all]{xy}

\renewcommand{\marginpar}

\definecolor{NoteColor}{rgb}{1,0,0}

\newcommand{\note}[1]{\textcolor{NoteColor}{#1}}

\newtheorem{theorem}{\rm\bf Theorem}[section]
\newtheorem{main}{\rm\bf Theorem}
\newtheorem{proposition}[theorem]{\rm\bf Proposition}
\newtheorem{lemma}[theorem]{\rm\bf Lemma}

\newtheorem{corollary}[theorem]{\rm\bf Corollary}
\newtheorem*{theorem 1}{\rm\bf Proposition 1}
\newtheorem*{theorem 2}{\rm\bf Proposition 2}

\theoremstyle{definition}

\theoremstyle{remark}
\newtheorem{remark}[theorem]{\rm\bf Remark}
\newtheorem{questions}[theorem]{\rm\bf Questions}

\def\interieur#1{\mathord{\mathop{\kern 0pt #1}\limits^\circ}}

\title[Horofunction compactification]{The horofunction compactification of Teichm\"uller spaces of surfaces with boundary}\thanks{This
work was started while the authors were visiting Mathematisches Forschungsinstitut Oberwolfach from July 08 to July 21, 2012.
We appreciate the support of  MFO.  L. Liu is partially supported by NSFC No: 11271378; W. Su is partially supported by NSFC No: 11201078.
 All the authors are partially supported by the French ANR grant FINSLER. The author would like to thank the referee for his (or her) corrections and useful comments.}

\author{D. Alessandrini}
\address{Daniele Alessandrini, Universit\"at Heidelberg, Mathematisches Institut, INF 205,
69120 Heidelberg}
\email{daniele.alessandrini@gmail.com}

\author{L. Liu}
\address{Lixin Liu, Department of Mathematics, Sun Yat-Sen University, 510275, Guangzhou, P. R. China}
\email{mcsllx@mail.sysu.edu.cn}

\author{A. Papadopoulos}
\address{Athanase Papadopoulos,  Institut de Recherche Math\'ematique Avanc\'ee, Universit{\'e} de Strasbourg and CNRS,
7 rue Ren\'e Descartes,
 67084 Strasbourg Cedex, France and CUNY, Hunter College, Department of Mathematics  and Statistics,  695, park Ave. NY 10065, USA} \email{athanase.papadopoulos@math.unistra.fr}
\date{\today}

\author{W. Su}
\address{Weixu Su, School of Mathematics and Shanghai Mathematics Center, Fudan University, 200433, Shanghai, China.}
\email{suwx@fudan.edu.cn}

\date{\today}

\begin{document}

\begin{abstract}
The arc metric is an asymmetric metric on the Teichm\"uller space
$\mathcal{T}(S)$ of a surface $S$ with nonempty boundary.
It is the analogue of Thurston's metric on the Teichm\"uller space of a surface without boundary.
In this paper we study the relation between Thurston's compactification and the horofunction compactification
of $\mathcal{T}(S)$ endowed with the arc metric.
We prove that there is a natural homeomorphism between the
two  compactifications.
This generalizes a result of Walsh \cite{Walsh} that concerns Thurston's
 metric.
\medskip

\noindent The final version of this paper will appear in \emph{Topology and its Applications}.
\end{abstract}

\maketitle

%\bigskip

\medskip

\medskip

\maketitle
\tableofcontents
\section{Introduction}
  Let $R$ be an oriented surface of genus $g$ with $n$ punctures and let $\mathcal{T}_{g,n}$ be the
  Teichm\"uller space of $R$. We shall view $\mathcal{T}_{g,n}$ as a space of equivalence classes of metrics on $R$.
  Thurston introduced a compactification of $\mathcal{T}_{g,n}$,
  which is used in his classification of diffeomorphisms of surfaces \cite{Thurston88}.
The boundary of this compactification is the space of projective classes of measured foliations
on $R$. The action of the mapping class group on $\mathcal{T}_{g,n}$ extends continuously to Thurston's boundary.

  There is an asymmetric Finsler metric on $\mathcal{T}_{g,n}$ defined by Thurston \cite{Thurston}.
  The geodesics of this metric are families of extremal Lipschitz maps between hyperbolic surfaces.
  The space $\mathcal{T}_{g,n}$ endowed with Thurston's  metric is a complete (asymmetric) geodesic metric space.
Unlike the classical Teichm\"uller metric, Thurston's metric is not uniquely geodesic.
 A special kind of geodesics for this
 metric, called \emph{stretch lines}, are constructed by ``stretching''  along
complete geodesic laminations of hyperbolic surfaces, that is, geodesic
laminations whose complementary regions are all ideal triangles.
The introduction of this metric paved the way to a whole set of new interesting questions on the geometry of Teichm\"uller space \cite{PT,Su}.

Thurston's compactification and Thurston's metric are closely related to each other.
The connection between
Thurston's compactification and the geodesic rays of Thurston's  metric was shown by
Papadopoulos \cite{Papa}.  To state things more precisely, let $\mu$ be a complete geodesic lamination.
(Note that we do not assume that $\mu$ carries a transverse invariant measure of full support.)
 Associated to $\mu$ is a global parametrization of
$\mathcal{T}_{g,n}$, called the \emph{cataclysm coordinates},  sending $\mathcal{T}_{g,n}$ to the set of measured foliations
transverse to $\mu$.
The cataclysm coordinates  extend continuously
to Thurston's boundary (see \cite[Theorem 4.1]{Papa} for a more precise statement).
In particular, a stretch line is determined by a measured foliation $F$ that is transverse to $\mu$ (called the horocyclic foliation associated to the stretch line) and this stretch line converges to the projective class of $F$ in Thurston's boundary \cite{Papa}.

 Walsh \cite{Walsh} showed that
Thurston's compactification of $\mathcal{T}_{g,n}$ can be naturally identified with the
horofunction compactification with respect to Thurston's metric.
Horofunction boundaries have the property that each geodesic ray converges to a point on the boundary.
As a corollary, every geodesic ray for Thurston's metric
converges to a point in Thurston's boundary.

Another direct corollary of
the result of Walsh \cite{Walsh} is that
any isometry of $\mathcal{T}_{g,n}$ equipped with Thurston's metric induces
a self-homeomorphism of Thurston's boundary. On the other hand, there is a ``detour cost'' distance (which is also asymmetric and which may take the value infinity) defined on Thurston's boundary
which is preserved by the isometries of $\mathcal{T}_{g,n}$, equipped with Thurston's metric. By calculating the detour cost between any two
projective measured foliations,
Walsh \cite{Walsh} proved that, with some exceptional cases, the isometry group of $\mathcal{T}_{g,n}$ equipped with Thurston's
metric is the extended mapping class group.

\bigskip

In this paper, we compare Thurston's compactification of  the Teichm\"uller spaces
of surfaces with boundary with the horofunction boundary of that space with respect to an appropriate metric, the \emph{arc metric} introduced in \cite{LPST}.

Thurston's asymmetric metric  can be defined by a formula which compares lengths of simple closed curves computed with the metrics representing the two elements in Teichm\"uller space (\S \ref{sec:metric}). The passage to the definition of the arc metric, using lengths of arcs, is very natural, but there are geometric questions whose solutions are far from obvious. We mention for instance that it is unknown whether the arc metric is Finsler, or whether it realizes the extremal Lipschitz constant of homeomorphisms between hyperbolic surfaces, as in the case of Thurston's metric on Teichm\"uller spaces of surfaces without boundary. We also do not know whether two points in Teichm\"uller space of a surface with boundary are joined by a concatenation of stretch lines. Working with arcs on a surface with boundary, instead of simple closed curves, involves several complications and requires new topological and geometrical tools, and this makes this subject interesting.

We now present our results in more detail.

Let $S$ be a hyperbolic surface of finite area with
totally geodesic boundary components and
let $\mathcal{T}(S)$ be the Teichm\"uller space of $S$.
There is an analogue of Thurston's compactification of $\mathcal{T}(S)$  defined using hyperbolic length and intersection
number with simple closed curves and
simple arcs on $S$ (see \S \ref{sec:boundary}). The boundary of such a compactification is identified with
the space of projective measured laminations on $S$,
which is homeomorphic to a sphere (see Theorem \ref{thm:double} and Proposition \ref{pro:sphere}).

We recall the definition of the arc metric in  \S \ref{sec:arc} and prove the following:

\begin{main}\label{thm:main}
Thurston's compactification of $\mathcal{T}(S)$ can be identified with the horofunction compactification of the arc metric on
$\mathcal{T}(S)$ by a natural homeomorphism.
\end{main}

The proof of Theorem \ref{thm:main} depends on the study of the asymptotic behaviour of
the geodesic lengths of simple closed curves and arcs along certain paths on
$\mathcal{T}(S)$. In particular, for every measured lamination $\mu$, we will construct (Lemma \ref{lemma:key}) a path $X_t, t\in [0,+\infty)$ in $\mathcal{T}(S)$  such that
each simple closed curve or simple arc $\alpha$ on $S$ satisfies
$$e^{t} i(\mu,\alpha)-C \leq \ell_\alpha(X_t) \leq e^{t} i(\mu,\alpha)+C_\alpha,$$
where $C>0$ is a uniform constant and $C_\alpha>0$ is a constant depending on $\alpha$.

\begin{remark}It is reasonable to conjecture that,
in the case where $S$ has nonempty boundary, the isometry group of $\mathcal{T}(S)$ endowed with the arc metric
is the (extended) mapping class group $\mathrm{Mod}(S)$, with the usual exceptional surfaces that appear in the theory without boundary. 
As a matter of fact, if $S^d= S\cup \bar{S}$ be the double of $S$,
obtained by taking the mirror image $\bar{S}$ of $S$ and by identifying the corresponding boundary components by
an orientation-reversing homeomorphism,
then $S^d$ is a surface without boundary. We know that such a doubling induces an isometric embedding from
$\mathcal{T}(S)$  to $\mathcal{T}(S^d)$ (see \S \ref{sec:pre}). As a result, one may hope that Walsh's argument
can be applied.
However, the proof of Walsh depends on Thurston's construction of stretch maps,
which does not apply as such to $\mathcal{T}(S)$ when the surface $S$ has boundary components.
A further understanding of Thurston's compactification of $\mathcal{T}(S)$ and the action of isometry group
may require some generalized notion of (appropriately defined) ``stretch map" for surfaces with boundary.
\end{remark}

\section{Preliminaries}\label{sec:pre}

Throughout this paper, we denote by $S=S_{g,n,p}$ a connected orientable surface of finite type, of genus $g$ with $n$
punctures and  $p$  boundary components. We always assume that the Euler characteristic $\chi(S)=2-2g-n-p$ is $<0$ and that
the boundary of $S$, denoted by $\partial S$, is nonempty.

A \emph{hyperbolic structure} on $S$ is a complete metric of constant curvature $-1$ such that
\begin{enumerate}[(i)]
\item each puncture has a neighborhood which is isometric to a cusp, i.e., to
the quotient of $\{z=x+iy \in\mathbb{H}^{2}\  | \ y>a \}$, for some $a>0$,
by the group generated by the translation $z\mapsto z+1$;
\item each boundary component is a closed geodesic.
\end{enumerate}

A marked hyperbolic surface is a pair $(X, f)$, where
$X$ is a hyperbolic structure on $S$ and $f:S\to X$ an orientation-preserving homeomorphism.
 The map $f$ is called a \emph{marking}.
Two marked hyperbolic surfaces $(X_1, f_1)$ and $(X_2, f_2)$ are said to be equivalent if
there exists an isometry $h: X_1\to X_2$ which is homotopic
to $f_2 \circ f_1^{-1}$ (note that in our setting, homotopies fix each boundary component setwise but they do not need to fix it pointwise).
The \emph{reduced Teichm\"uller space} $\mathcal{T}(S)$  is the set of equivalence classes of
marked hyperbolic structures on $S$.

\begin{remark}
 Since all
Teichm\"uller spaces that we consider are reduced, we shall omit the word ``reduced"
in our exposition. Furthermore, we shall sometimes denote an equivalence class of $(X, f)$ in $\mathcal{T}(S)$ by $X$,
without explicit reference to the marking or to the equivalence relation.
\end{remark}

Let $S^d$ be the double of $S$ and $\mathcal{T}(S^d)$ the Teichm\"uller space of $S^d$.
Note that $S^d$ is a surface of genus $2g+p-1$ with $2n$ punctures, without boundary.
We construct a natural embedding of $\mathcal{T}(S)$ into $\mathcal{T}(S^d)$.

For any equivalence class of marked hyperbolic structures $[(X,f)]\in \mathcal{T}(S)$, we
let $\overline X$ be the isometric mirror image of $X$. The   hyperbolic surface $\overline X$ is equipped with an orientation-reversing
isometry $J: X\to \overline X$. Then $X^d$ is the hyperbolic surface
obtained by taking the disjoint union of $X$ and $\overline X$, and gluing
$\partial X$ with $\partial \overline X$ by the restriction of $J$ to the boundary. This map $J$ extends to an involution of $X^d$
which we still denote by $J$. Taking the double of a marked Riemann surface with boundary is a well-known operation, and it was already considered in Teichm\"uller's paper \cite{T20}. We are dealing here with the analogous operation, at the level of the associated hyperbolic structures. To determine a point in $\mathcal{T}(S^d)$, we have to choose a marking for $X^d$.
Note that we can modify the marking $f: S \to X$ in its homotopy class in such a way that
$f=\mathrm{id}$ in a small collar neighborhood of each boundary component.
We extend $f$ to a marking
$$\tilde{f}:S^d \to X^d$$
by setting

$$\tilde{f}(x)=   J \circ f \circ J(x)$$
when $x\in \overline X$. It is easy to check that the equivalence class $[(X^d, \tilde{f})]$ is independent of
the choice of $(X,f)\in [(X,f)]$.

We set $\Psi([(X,f)])=[(X^d, \tilde{f})]$ and we use for simplicity the notation
$\Psi(X)=X^d$. Then we have

\begin{proposition}
The map
\begin{eqnarray*}\label{eq:X}
\Psi: \mathcal{T}(S) &\to& \mathcal{T}(S^d), \\
X  &\mapsto& \Psi(X)=X^d.
\end{eqnarray*}
is an embedding.
\end{proposition}
\begin{proof}
An efficient way to see that $\Psi$ is an embedding is
to present $\Psi$ in terms of Fenchel-Nielsen coordinates.
We choose a maximal set $\{\alpha_i\}_{i=1}^{3g-3+n+p}$ of mutually disjoint and non homotopic simple closed curves in
the interior of $S$, all of them non-trivial and not homotopic to boundary components.
Denote the boundary components of $S$ by $\{\beta_j\}_{j=1}^{p}$.
The map

\begin{eqnarray*}
 \mathcal{T}(S) &\to& (\mathbb{R}_+\times \mathbb{R})^{3g-3+p+n}\times (\mathbb{R}_+)^{p}, \\
X  &\mapsto& \big(\ell_{\alpha_i}(X),\tau_{\alpha_i}(X)\big)\times \ell_{\beta_j}(X),
\end{eqnarray*}
where $\ell_{\alpha_i},\ell_{\beta_j}$ are the length coordinates and $\tau_{\alpha_i}$ the twist coordinates,
defines the Fenchel-Nielsen coordinates of $\mathcal{T}(S)$ (see Buser \cite{Buser}).

For each $1\leq i\leq 3g-3+n+p$, let $\bar{\alpha}_i\subset \overline{S}$ be the mirror image of $\alpha_i$.
Then $\{\alpha_i\}\cup \{\beta_j\}\cup\{\bar{\alpha}_i\}$ is a pants decomposition of $S^d$.
Denote the Fenchel-Nielsen coordinates of $\mathcal{T}(S^d)$ by
$$(\ell_{\alpha_i},\tau_{\alpha_i})\times (\ell_{\beta_j},\tau_{\beta_j})\times (\ell_{\bar{\alpha}_i},\tau_{\bar{\alpha}_i}).$$
Then the map $\Psi$ can be written in the Fenchel-Nielsen coordinates as

\begin{equation}\label{eq:embed}
(\ell_{\alpha_i},\tau_{\alpha_i})\times \ell_{\beta_j} \mapsto
(\ell_{\alpha_i},\tau_{\alpha_i})\times (\ell_{\beta_j},0)\times
(\ell_{{\alpha}_i},-\tau_{{\alpha}_i}).
\end{equation}
Note that $\tau_{\bar \alpha_i}=-\tau_{\alpha_i}$ since the mirror image of a right twist deformation
on $X$ becomes a left twist deformation on $\overline X$.

Since the Fenchel-Nielsen coordinates are real-analytic global parameters for Teichm\"uller spaces,
and the map \eqref{eq:embed} is a real-analytic embedding, it follows that
$\Psi$ gives a real-analytic embedding of $\mathcal{T}(S)$ into $\mathcal{T}(S^d)$.
\end{proof}

The map $\Psi$ will be an isometric embedding if we equip $\mathcal{T}(S)$ with the arc metric and
$\mathcal{T}(S^d) $ with Thurston's metric \cite{LPST}.
We shall recall the definition of Thurston's metric  in \S \ref{sec:metric}  and  the arc metric in \S \ref{sec:arc}.

We consider the involution $J:S^d\to S^d$ on $\mathcal{T}(S^d)$ as an element of the extended mapping class group (that is, we identity when needed a map with its homotopy class).
We set
$$\mathcal{T}^{sym}(S^d):=\{R \in \mathcal{T}(S^d) \ | \ J(R)=R\}.$$
It is not hard to see that there is a canonical identification $\Psi (\mathcal{T}(S))\simeq \mathcal{T}^{sym}(S^d)$.

\section{Measured laminations and Thurston's compactification}\label{sec:boundary}
In this section, we recall the notion of measured lamination space and the Thurston compactification of
Teichm\"uller space, and their extensions
to hyperbolic surfaces with geodesic boundaries. Part of our results here is a continuation of work done in \cite{LPST}.
\subsection{Measured laminations}
In the setting of surfaces with boundary,
we need to be precise on the definition of measured geodesic laminations that we deal
with.

We endow $S$ with a fixed hyperbolic structure. A \emph{geodesic lamination} $\lambda$ on $S$
is a closed subset of $S$ which is the union of disjoint simple geodesics called
the \emph{leaves} of $\lambda$.
With such a definition, a leaf $L$ of $\lambda$ may be a  boundary component of $S$. It may also be a geodesic
ending at a cusp or  a boundary component of $S$. Furthermore, $L$ may meet a boundary component of $S$ or spiral along it.
If $L$ is a geodesic with some end at a point $p\in \partial S$, we require that
 $L$ is perpendicular to $\partial S$ at $p$.

Let $\lambda$ be a geodesic lamination on $S$ with compact support.
A \emph{transverse measure} for $\lambda$
is an assignment, for each embedded arc $k$ on $S$ which is transverse to $\lambda$ and with endpoints
contained in the complement of $\lambda$,
of a finite Borel measure $\mu$ on $k$ with the following properties:

\begin{enumerate}
\item The support of $\mu$ is $\lambda\cap k$.
\item For any two transverse arcs $k$ and $k'$ that are homotopic through embedded arcs which move
their endpoints within fixed complementary components of $\mu$,  the assigned measures satisfy
$$\mu(k)=\mu(k').$$
\end{enumerate}

A \emph{measured geodesic lamination} is a geodesic lamination $\lambda$
together with a transverse measure.
To simplify notation, we shall sometimes talk about a ``measured lamination"
instead of a ``measured geodesic lamination".
We shall denote such a measured lamination by $(\lambda,\mu)$ or, sometimes, $\mu$ for simplicity.

All the measured laminations we consider are assumed to have compact support.
An example of a measured lamination is a weighted simple closed
geodesic, that is,  a simple closed geodesic $\alpha$ equipped with a positive weight $a > 0$.
The measure disposed on a transverse arc $k$ is then the sum of the Dirac masses at the
intersection points between $k$ and $\alpha$ multiplied by the weight $a$.
In general, a lamination is a finite union of uniquely defined minimal  sub-laminations,
 called its \emph{components}. With the assumptions we made, each such component is of one of the following
three types:
\begin{enumerate}[(i)]
\item a simple closed geodesic in $S$ (such a simple closed geodesic can be a boundary component);
\item a geodesic arc meeting $\partial S$ at right angles;
\item a measured geodesic lamination in the interior of $S$, in which every leaf is
dense.
\end{enumerate}
This follows from our definition and from the corresponding result for
surfaces without boundary.

 Let $\mathcal{ML}(S)$ be the space of measured geodesic laminations on $S$.
 We shall equip $\mathcal{ML}(S)$  with the $\mathrm{weak}^\ast$-topology,
 following Thurston \cite{Thurston-notes} in the case of surfaces without boundary.
We can choose a finite collection of generic geodesic arcs $k_1,\cdots, k_m$ on $S$
such that $\mu_n\in \mathcal{ML}(S)$ converges to $\mu$ if and only if
$$\max_{i=1,\cdots,m}\big| \int_{k_i} d\mu_n - \int_{k_i}d\mu \big|\to 0.$$
Here a geodesic arc is called \emph{generic} if it is transverse to any simple geodesic
on $S$. Note that almost every geodesic arc on $S$ is generic \cite{Bonahon}.

 We also recall that there are natural
homeomorphisms between the various measured lamination spaces when the hyperbolic structure
on the surface varies. Using this fact, it is possible to talk about a measured geodesic
lamination on the surface without referring to a specific hyperbolic structure on it.

\bigskip

Let $S^d$ be the double of $S$ and $\mathcal{ML}(S^d)$ be
the space of measured geodesic laminations on $S^d$. As before, denote the natural involution
of $S^d$ by $J$. For any subset $A\subset S$ or $A\subset \overline S$, we denote by $\bar A=J(A)$.
Moreover, if $\mu$ is a measure on an arc $I$ on $S$ or $\overline S$, then we set
$\bar\mu(I)=\mu(J(I))$.
From the above definition of  measured geodesic lamination on $S$,
there is
a natural inclusion $\psi$ from $\mathcal{ML}(S)$ into the space $\mathcal{ML}(S^d)$
defined by

\begin{eqnarray*}
\psi:  \mathcal{ML}(S) &\to&  \mathcal{ML}(S^d) \\
(\lambda,\mu)&\mapsto& (\lambda\cup\bar\lambda, \mu+\bar\mu).
\end{eqnarray*}
We will use the notation $\mu^d= (\lambda\cup\bar\lambda, \mu+\bar\mu)$ for simplicity.
Note that if $\mu$ is a weighted simple closed
geodesic $(\alpha,a)$ where $\alpha$ is a boundary component of $S$ and $a$ the weight it carries,
then $\mu^d=(\alpha,2a)$.

A measured lamination (respectively, hyperbolic structure, simple closed curve,
etc.) on $S^d$ is said to be \emph{symmetric} if it is invariant by the canonical involution $J$.
Denote the subset of all symmetric measured laminations in $\mathcal{ML}(S^d)$
by $\mathcal{ML}^{sym}(S^d)$.
\begin{lemma}\label{lem:sym}
The map $\psi:\mathcal{ ML}(S)\to \mathcal{ML}(S^d)$ is continuous and we have a natural identification  $$\mathcal{ML}^{sym}(S^d)=\psi(\mathcal{ML}(S)).$$
\end{lemma}
\begin{proof}

It is obvious that all elements in $\psi(\mathcal{ML}(S))$ are symmetric.

Conversely, let $\widetilde{\mu}$ be a symmetric measured lamination in $\mathcal{ML}(S^d)$.
Every component of $\mu$
which meets the fixed point locus of the involution $J$ is, if it exists, a simple closed
geodesic. Indeed, such a component must intersect the fixed point locus perpendicularly,
and no component which is not a simple closed geodesic can intersect the fixed
point locus in this way, because of the recurrence of leaves.
It follows that any sublamination of
$\widetilde{\mu}$ which is connected (that is, which has only one component) that intersects $\partial S$
is either a boundary component of $S$ or a symmetric closed geodesic meeting $\partial S$ at right angles.
As a result, the restriction of $\widetilde{\mu}$ to $S$ defines a unique measured lamination $\mu\in\mathcal{ML}(S)$
such that $\widetilde{\mu}=\psi(\mu)=\mu^d$.

The continuity of $\psi$ follows directly from the definition of the $\mathrm{weak}^\ast$-topology
on measured lamination spaces.
\end{proof}
\subsection{Rational measured laminations are dense in $\mathcal{ML}(S)$.}
We say that a simple closed curve on a surface  is \emph{essential} if it is neither
 homotopic to a puncture nor homotopic to a point (but it can be homotopic to a boundary component).
We let $\mathcal{C}(S)$ be the set of homotopy classes of essential simple closed curves
on $S$.

In the case where $\partial S$ is nonempty, an \emph{arc} in $S$ is the homeomorphic image of a closed interval which is properly
embedded in $S$ (that is, the interior of the arc is in the interior of $S$ and the endpoints
of the arc are on the boundary of $S$). All homotopies of arcs that we consider are
relative to $\partial S$, that is, they keep the endpoints of arcs on the set $\partial S$ (but they do not
necessarily fix pointwise the points on $\partial S$). An arc is said to be \emph{essential}
if it is not homotopic to a subset of $\partial S$.
We let $\mathcal{A}(S)$ be the set of homotopy classes of essential arcs on $S$.

Endowing $S$ with a hyperbolic structure $X$, for any $\gamma\in \mathcal{A}(S)\cup \mathcal{C}(S)$,
there is a unique geodesic $\gamma^X$ in its homotopy class. It is orthogonal to $\partial X$ at each intersection point, in the case where
$\gamma$ is an equivalence class of arc. We denote by
$\ell_\gamma(X)$ the length of
$\gamma^X$, and we call it the \emph{geodesic length} of $\gamma$ on $X$. This length only depends on
the equivalence class of $X$ in Teichm\"uller space.

The geodesic representation $\gamma \mapsto \gamma^X$ defines a
correspondence between $\mathbb{R}_+\times \big( \mathcal{A}(S)\cup \mathcal{C}(S)\big)$
and the set of  weighted simple closed geodesics and weighted simple geodesic arcs on $S$.

\bigskip

A measured lamination $\mu$ is \emph{rational} if the support of $\mu$ consists of simple closed geodesics
or simple geodesic arcs. Let us denote a rational measured lamination by
$$\sum_{i\in\mathcal{I}} a_i \gamma_i,$$
where $\mathcal{I}$ is some finite set, $a_i>0$ and the $\gamma_i\in \mathcal{A}(S)\cup \mathcal{C}(S)$ are pairwise disjoint.

The set of weighted simple closed curves on $S^d$ is dense in the
space $\mathcal{ML}(S^d)$, and the geodesic length function, defined on weighted simple
closed geodesics, extends to a continuous function on the space $\mathcal{ML}(S^d)$
\cite{Thurston-notes}. The situation is slightly different for surfaces with boundary.

In general,  the set $\mathbb{R}_+\times \mathcal{A}(S)\cup \mathbb{R}_+\times \mathcal{C}(S) $
is not dense in $\mathcal{ML}(S)$. For example, if $\mu=\alpha+\beta$ where $\alpha$ is a simple closed curve
in the interior of $S$ and $\beta$ is a boundary component of $S$, then $\mu$ cannot be approximated by
a sequence in $\mathbb{R}_+\times \mathcal{A}(S)\cup \mathbb{R}_+\times \mathcal{C}(S)$. However, using multiple curves and arcs instead of curves and arcs suffices, and we have the following:

\begin{lemma}\label{lem:dense}
The set of  rational measured laminations on $S$ is dense in $\mathcal{ML}(S)$.
\end{lemma}
\begin{proof}
Let $\mu\in \mathcal{ML}(S)$.
Each component of $\mu$ is either a simple closed geodesic, 
a geodesic arc or a minimal measured lamination in the interior of $S$.
Each minimal component is contained
in a geodesically convex subsurface, whose interior is disjoint from the other components.
Thus, by Thurston's theory,  each minimal component can be approximated by a sequence of weighted simple closed geodesics
in the interior of the subsurface.
\note{}
\end{proof}

\begin{proposition}\label{pro:maximum}
For every $X$ and $Y$ in $\mathcal{T}(S)$, we have
\begin{equation}\label{equ:maximum}
\sup_{\gamma\in \mathcal{C}(S)\cup \mathcal{A}(S)}  \frac{\ell_\gamma(Y)}{\ell_\gamma(X)}= \sup_{\mu\in \mathcal{ML}(S)}  \frac{\ell_\mu(Y)}{\ell_\mu(X)}.
\end{equation}
\end{proposition}
\begin{proof}
It is obvious that
$$ \sup_{\gamma\in \mathcal{C}(S)\cup \mathcal{A}(S)}  \frac{\ell_\gamma(Y)}{\ell_\gamma(X)}\leq \sup_{\mu\in \mathcal{ML}(S)}  \frac{\ell_\mu(Y)}{\ell_\mu(X)}.
$$

Let us set $$\mathcal{ML}_1(S)=\{\mu\in \mathcal{ML}(S) \ | \ \ell_\mu(X)=1\}$$
and $$\mathcal{ML}_2(S^d)=\{\tilde{\mu}\in \mathcal{ML}(S^d) \ | \ \ell_{\tilde{\mu}}(X)=2\}.$$
The map $\psi$ sends $\mathcal{ML}_1(S)$ into $\mathcal{ML}_2(S^d)$.

Since $\mathcal{ML}_2(S^d)$ is compact and $\psi(\mathcal{ML}_1(S))$ is a closed subset of
$\mathcal{ML}_2(S^d)$, $\mathcal{ML}_1(S)$ is a compact subset of $\mathcal{ML}(S)$. Therefore, there is a measured lamination $\mu_0\in \mathcal{ML}_1(S)$
that realizes the maximum:
\begin{equation}\label{eq:sK}
\sup_{\mu\in \mathcal{ML}(S)}  \frac{\ell_\mu(Y)}{\ell_\mu(X)}=\frac{\ell_{\mu_0}(Y)}{\ell_{\mu_0}(X)}.
\end{equation}

Consider the decomposition of $\mu_0$
 into minimal
components,
$$\mu_0= \sum_i a_i\nu_i.$$ Let $K$ be the value of the supremum in $(\ref{eq:sK})$.  We
have $\ell_{\mu_0}(Y)=K\ell_{\mu_0}(X)$, that is, since the length function is positively homogeneous,
$$\sum_i a_i\ell_{\nu_i}(Y)=K \sum_i a_i\ell_{\nu_i}(X).$$
Since $\ell_{\nu_i}(Y)\leq K\ell_{\nu_i}(X)$ (from the definition), it follows that
$$\ell_{\nu_i}(Y)= K \ell_{\nu_i}(X)$$
for each $\nu_i$. As a result, any component of $\mu_0$
 also realizes the supremum $L$.

 As before, since each component of $\mu_0$ is either a simple closed geodesic, a geodesic arc or a minimal measured lamination in the interior of $S$, each of which can be approximated by a sequence in
$\mathbb{R}_+\times \big( \mathcal{A}(S)\cup \mathcal{C}(S)\big)$, we conclude that
$$\sup_{\gamma\in \mathcal{C}(S)\cup \mathcal{A}(S)}  \frac{\ell_\gamma(Y)}{\ell_\gamma(X)}= \sup_{\mu\in \mathcal{ML}(S)}  \frac{\ell_\mu(Y)}{\ell_\mu(X)}.$$
\end{proof}

Denote by $\mathcal{B}$ the set of all boundary components of $S$. In the paper \cite{LPST},
the following was shown:
\begin{proposition}\label{prop:equality}
$$\sup_{\gamma\in \mathcal{C}(S)\cup \mathcal{A}(S)}  \frac{\ell_\gamma(Y)}{\ell_\gamma(X)}=\sup_{\gamma\in \mathcal{B}(S)\cup \mathcal{A}(S)}  \frac{\ell_\gamma(Y)}{\ell_\gamma(X)}\geq 1,$$
and the last inequality becomes an equality if and only if $X=Y$.
\end{proposition}

\subsection{Thurston's compactification}

We need to recall some fundamental
results of Thurston described in \cite{FLP}.

Let $R$ be a surface
 of genus $g$ with $n$ punctures. Let $\mathbb{R}_+^{\mathcal{C}(R)}$ be the set of all nonnegative functions on $\mathcal{C}(R)$ and $P\mathbb{R}_+^{\mathcal{C}(R)}$ the projective space of
$\mathbb{R}_+^{\mathcal{C}(R)}$ (that is, its quotient by the action of positive reals). We denote by $\pi: \mathbb{R}_+^{\mathcal{C}(R)} \to P\mathbb{R}_+^{\mathcal{C}(R)}$
the natural projection. We endow $\mathbb{R}_+^{\mathcal{C}(R)}$ with the product topology and $P\mathbb{R}_+^{\mathcal{C}(R)}$ with the quotient  topology.
There is a mapping $L$ defined by

\begin{eqnarray*}
L: \mathcal{T}(R) &\to& \mathbb{R}_+^{\mathcal{C}(R)},\\
X&\to& (\ell_\alpha(X))_{\alpha\in \mathcal{C}(R)}.
\end{eqnarray*}
The map
$\pi \circ L: \mathcal{T}(R) \to P\mathbb{R}_+^{\mathcal{C}(R)}$ is an embedding.

There is also a mapping $I$, defined by
\begin{eqnarray*}
I: \mathcal{ML}(R) &\to& \mathbb{R}_+^{\mathcal{C}(R)},\\
\mu &\mapsto& (i(\mu,\alpha))_{\alpha\in \mathcal{C}(R)},
\end{eqnarray*}
where $$i(\mu,\alpha)=\inf_{\alpha'\in [\alpha]}\int_{\alpha'} d\mu$$
is the intersection number. Then $I$ is also an embedding.

Thurston showed that
the closure of
$\pi \circ L(\mathcal{T}(R))$ is compact and coincides with $$\pi \circ L(\mathcal{T}(R)) \ \cup \  \pi \circ I (\mathcal{ML}(R)).$$
We denote this closure by $\overline{\mathcal{T}(R)}$. This is \emph{Thurston's compactification}
of $\mathcal{T}(R)$.
In the following,
we shall identify $\mathcal{T}(R)$ with its image and the boundary of
$ \overline{\mathcal{T}(R)}$ with $\mathcal{PML}(R)$, the space of projective classes of measured laminations
on $R$.

Now we introduce an analogue of Thurston's compactification for the Teichm\"uller space $\mathcal{T}(S)$,
where $S$ is a surface with boundary.
 For simplicity, let $\mathcal{C}=\mathcal{C}(S)$ and $\mathcal{A}=\mathcal{A}(S)$.

 Consider the map defines by the following composition:
\begin{eqnarray}\label{equ:map1}
\mathcal{T}(S) \stackrel{L}{\longrightarrow} \mathbb{R}_+^{\mathcal{C}\cup \mathcal{A}}
\stackrel{\pi}{\longrightarrow}  P\mathbb{R}_+^{\mathcal{C}\cup \mathcal{A}}.
\end{eqnarray}
\begin{lemma}\label{lem:map1}
The map defined in $(\ref{equ:map1})$ is injective.
\end{lemma}
\begin{proof}
Suppose that $X, Y\in \mathcal{T}(S)$ are mapped to the same point in $P\mathbb{R}_+^{\mathcal{C}\cup \mathcal{A}}$.
Then there exists a constant $K>0$ such that
$$\ell_\gamma(X)=K\ell_\gamma(Y)$$
for all $\gamma\in \mathcal{C}\cup \mathcal{A}$.
Without loss of generality, we may assume that $K\geq 1$.
This implies that

$$\sup_{\gamma\in \mathcal{C}(S)\cup \mathcal{A}(S)}  \frac{\ell_\gamma(Y)}{\ell_\gamma(X)}\leq 1.$$
It follows from Proposition \ref{prop:equality} that $X=Y$.
\end{proof}

Similarly, we consider
\begin{eqnarray}\label{equ:map2}
\mathcal{ML}(S) \stackrel{I}{\longrightarrow} \mathbb{R}_+^{\mathcal{C}\cup \mathcal{A}}.
\end{eqnarray}

\begin{lemma}\label{lem:map2}
The map defined in $(\ref{equ:map2})$ is injective.
\end{lemma}
\begin{proof}
Suppose that $\mu,\nu\in \mathcal{ML}(S)$ are mapped to the same point in $\mathbb{R}_+^{\mathcal{C}\cup \mathcal{A}}$.
Let
$$\mu=\mu_0+\mu_1, \nu=\nu_0+\nu_1,$$
where $\mu_0,\nu_0$ are unions of components contained in the interior of $S$
and $\mu_1,\nu_1$ are unions of components that belong to $\partial S$ or intersect $\partial S$ (some of these components might be empty).

Since $\mu_1$ and $\nu_1$ (if they exist) are unions of  simple geodesic arcs or boundary components of
$S$, it is easy to see that $\mu_1=\nu_1$. For otherwise, it is not hard to find  some element $\gamma$ in
$\mathcal{C}\cup \mathcal{A}$ such that $i(\mu_1,\gamma)\neq i(\nu_1,\gamma)$.

On the other hand, since $\mu_0$ and $\nu_0$ are contained in the interior of the surface, by the same argument as
for a surface without boundary which may have punctures, we have $\mu_0=\nu_0$.

 It follows that $\mu=\nu$.
\end{proof}

\begin{remark}\label{remark:disjoint}
Both Lemmas \ref{lem:map1} and \ref{lem:map2} can be proved directly by the same arguments as
\cite[Expos\'es 6 and 7]{FLP}. Note that the images of $\mathcal{T}(S)$ and $\mathcal{ML}(S)$
in $\mathbb{R}_+^{\mathcal{C}\cup \mathcal{A}}$ are disjoint. This follows from the fact that for
each $X\in \mathcal{T}(S)$, the set of lengths $\ell_\gamma(X),\gamma\in \mathcal{C}\cup \mathcal{A}$
is bounded below by a strictly positive constant (only depending on $X$); while for each $\mu\in \mathcal{ML}(S)$
and for any $\epsilon>0$, there is some $\gamma\in \mathcal{C}\cup \mathcal{A}$ such that
$$i(\mu,\gamma)<\epsilon.$$
Here $\gamma$ can be taken as a simple closed curve, a simple arc belonging to a component of $\mu$ (if it exists)
or a simple closed curve quasi-transverse to $\mu$ (see
\cite[Proposition 8.1]{FLP} for details).
\end{remark}

By Lemma \ref{lem:map1}, Lemma \ref{lem:map2} and Remark \ref{remark:disjoint}, we have an embedding

$$\mathcal{T}(S)\cup \mathcal{PML}(S)\to P\mathbb{R}_+^{\mathcal{C}\cup \mathcal{A}}.$$

We have already identified $\mathcal{T}(S)$ with the subset $\mathcal{T}^{sym}(S^d)$ of $\mathcal{T}(S^d)$
by the map $\Psi$ and $\mathcal{PML}(S)$ with the subset $\mathcal{PML}^{sym}(S^d)$ of $\mathcal{PML}(S^d)$ by the map $\psi$.
To give an idea of the image of $\mathcal{T}(S)\cup \mathcal{PML}(S)$ in $P\mathbb{R}_+^{\mathcal{C}\cup \mathcal{A}}$,
we shall show that the convergence of sequences in $\mathcal{T}(S)$ in the topology of $P\mathbb{R}_+^{\mathcal{C}\cup \mathcal{A}}$
is equivalent to the convergence in the topology of $P\mathbb{R}_+^{\mathcal{C}(S^d)}$.

 Let $\{X_n\}$ be a sequence in $\mathcal{T}(S)$  and let $\{X_n^d\}$ be the corresponding sequence in $\mathcal{T}^{sym}(S^d)$.

Assume that $X_n^d$ converges to a point $\widetilde{\mu} \in \mathcal{PML}(S^d)$ in the topology of $P\mathbb{R}_+^{\mathcal{C}(S^d)}$. Now an element
 of $\mathcal{T}(S^d)$ or $\mathcal{PML}(S^d)$ in $P\mathbb{R}_+^{\mathcal{C}(S^d)}$ is in $\mathcal{T}^{sym}(S^d)$ or $\mathcal{ML}^{sym}(S^d)$ if and only if as a function on the set of homotopy classes of curves $\mathcal{C}(S^d)$ it has the same values on pairs of curves that are images of each other by the involution $J$ of $S^d$. Thus, since $X^d_n$ is symmetric, $\widetilde{\mu}$ is also symmetric. It follows that $X_n$ converges to $\mu$ (which satisfies $\widetilde{\mu}=\mu^d$)
in the topology of $P\mathbb{R}_+^{\mathcal{C}\cup \mathcal{A}}$.

Conversely, assume that $X_n$ converges to a point $P$ in $P\mathbb{R}_+^{\mathcal{C}(S)\cup \mathcal{A}(S)}$.
Let $\widetilde{\mu}$ be any accumulation point of $X_n^d$ in $\mathcal{PML}^{sym}(S^d)$. By definition,
there exists a sequence $c_n>0$ such that (up to a subsequence)
$$c_n\ell_\gamma(X_n^d)\to i(\widetilde{\mu},\gamma)$$
 for any $\gamma\in \mathcal{C}(S^d)$. Setting $\bar\gamma=J(\gamma)$, we have
$$c_n\ell_{\bar\gamma}(X_n^d)=c_n\ell_\gamma(X_n^d)\to i(\mu,\gamma)=i(\mu,\bar\gamma).$$
In particular, we have
$$i(\widetilde{\mu},\gamma)=i(\widetilde{\mu},\bar\gamma)$$
for any $\gamma\in \mathcal{C}(S^d)$. Such a $\widetilde{\mu}$ must be symmetric and unique (the restriction of
$\widetilde{\mu}$ on $S$ is identified with $P$).

In conclusion, we have
\begin{theorem}\label{thm:double}
$\mathcal{PML}(S)$ is identified with the boundary of $\mathcal{T}(S)$ in $P\mathbb{R}_+^{\mathcal{C}\cup \mathcal{A}}$.
The embedding $\Psi: \mathcal{T}(S)\to \mathcal{T}(S^d)$ extends to
$\mathcal{T}(S)\cup \mathcal{PML}(S)$ such that $\Psi |_{\mathcal{PML}(S)}=\psi$.

\end{theorem}

\subsection{Topology of the boundary}

Let $S$ be a surface of genus $g$, with $p$ punctures and with $b$ boundary components denoted by $\{B_1, \dots, B_{b}\}$.
A  pants decomposition of $S$ contains $3g-3+b+p$ pairwise disjoint interior curves which we denote by $\{C_1, \dots, C_{3g-3+b+p}\}$,
 decomposing the surface into $2g-2+b+p$ pairs of pants.
Such a pants decomposition induces a symmetric  pants decomposition of the double $S^d$, with $6g-6+3b+2p$ curves denoted by $$\{C_1, \dots, C_{3g-3+b+p}, B_1, \dots, B_b,\bar{C}_1, \dots, \bar{C}_{3g-3+b+p}\},$$ dividing $S^d$ into $4g-4+2b+2p$ pairs of pants.

The space of measured laminations $\mathcal{ML}(S^d)$ can be understood using the Dehn-Thurston coordinates associated with a pants decomposition.

Given a measured lamination $\mu$,
for every curve $C$ in the symmetric pants decomposition of $S^d$, there are two associated coordinates, the length coordinate $i(\mu,C) \in \mathbb{R}_{\geq 0}$ and the twist coordinate $\theta(\mu,C) \in \mathbb{R}$ (see Dylan Thurston \cite{DT} for details). This gives an element $(i(\mu,C),\theta(\mu,C)) \in \mathbb{R}_{\geq 0} \times \mathbb{R}$. Consider the quotient $\mathbb{R}^{[2]} = \mathbb{R}_{\geq 0} \times \mathbb{R} / \sim $, where $(0,t) \sim (0,-t)$, and denote by $DT(\mu,C)$ the equivalent class of $(i(\mu,C),\theta(\mu,C))$ in $\mathbb{R}^{[2]}$. Notice that $\mathbb{R}^{[2]}$ is homeomorphic to $\mathbb{R}^{2}$. The Dehn-Thurston coordinates give a homeomorphism
\begin{eqnarray*}
 \mathcal{ML}(S^d) &\to& (\mathbb{R}^{[2]})^{6g-6+3b+2p}  \\
\mu &\rightarrow& ( DT(\mu,C_1), \dots, DT(\mu,C_{3g-3+b+p}),DT(\mu,B_1), \\
&& \dots, DT(\mu,B_b), DT(\mu,\bar{C}_1),\dots,DT(\mu,\bar{C}_{3g-3+b+p}) )
 \end{eqnarray*}

The subspace $\mathcal{ML}(S) \subset \mathcal{ML}(S^d)$ can be described by equations imposing symmetry on the coordinates:
$$\forall j: i(\mu,C_j) = i(\mu, \bar{C}_j)$$
$$\forall j: \theta(\mu,C_j) = -\theta(\mu, \bar{C}_j)$$
$$\forall j:   \theta(\mu, B_j) = 0 \mbox{ if }  i(\mu,B_j) \neq 0.$$
The minus sign in the equation for the twist comes from the fact that the sign of the twist parameter depends on the orientation of the surface,
and the mirror symmetry changes the orientation.

The first two equations mean that, for symmetric laminations, the coordinates associated with the curves $\bar{C}_i$ can be recovered from the coordinates associated to $C_i$, so we can neglect the curves $\bar{C}_i$ in the coordinates.

The third equation shrinks every factor $\mathbb{R}^{[2]}$ corresponding to a boundary curve $B_j$ into a line.
Then we define the coordinate $\hat{\theta}(\mu, B_j)$ as $i(\mu,B_j)$ if $i(\mu,B_j)\neq 0$, and as $-|\theta(\mu,B_j)|$ if $i(\mu,B_j) = 0$.

These considerations prove the following:

\begin{proposition}\label{pro:sphere}
The following map is a homeomorphism
$$ \mathcal{ML}(S) \ni \mu \rightarrow $$ $$(DT(\mu,C_1), \dots, DT(\mu,C_{3g-3+b+p}),\hat{\theta}(\mu,b_1), \dots, \hat{\theta}(\mu, B_b)) \in (\mathbb{R}^{[2]})^{3g-3+b+p}\times \mathbb{R}^b $$
In particular, $\mathcal{ML}(S)$ is homeomorphic to $\mathbb{R}^{6g-6+3b+2p}$, and $\mathcal{PML}(S)$ is homeomorphic to $\mathbb{S}^{6g-7+3b+2p}$.
\end{proposition}

\section{Thurston's asymmetric metric}\label{sec:metric}
Given a set $M$,  a nonnegative function $d$ defined on $M \times M$ is said to be a \emph{weak
metric} if it satisfies all the axioms of a distance function except the symmetry
axiom. A weak metric $d$ is said to be asymmetric if it is strictly weak, that is, if
there exist two points $x$ and $y$ in $M$ such that $d(x, y) \neq  d(y, x)$.

In this section, we first review Thurston's metric and stretch maps on
Teichm\"uller spaces of surfaces without boundary (with or without punctures).

Denote by $\mathcal{T}_{g,n}$ the Teichm\"uller space of a surface $R$ of genus $g$ with $n$ punctures and without boundary. The space $\mathcal{T}_{g,n}$ is the space of marked hyperbolic structures on $R$.
Thurston \cite{Thurston} defined an asymmetric metric $d_{\mathrm{Th}}$ on  $\mathcal{T}_{g,n}$ by setting
\begin{equation} \label{eq:L}
d_{\mathrm{Th}}(X,Y)=\inf_{f} \log L_f(X,Y),
\end{equation}
where the infimum is taken over all homeomorphisms $f:X \to Y$ homotopic to the identity map of $R$ and where $L_f(X,Y)$ is the \emph{Lipschitz constant} of $f$, that is,
\begin{displaymath}\label{Lip}
L_f(X,Y)=\sup_{x\neq y\in S}\frac{d_{Y}\big{(}f(x),f(y)\big{)}}{d_{X}\big{(}x,y\big{)}}.
\end{displaymath}

An important result of Thurston \cite{Thurston} is that

\begin{equation} \label{equ:Thurston}d_{\mathrm{Th}}(X,Y)=\log \sup_{\gamma\in \mathcal{C}(S)}  \frac{\ell_\gamma(Y)}{\ell_\gamma(X)}.
\end{equation}

The asymmetric metric defined in \eqref{eq:L} is Finsler, that is, it is a length metric which is defined by integrating a weak (asymmetric) norm on the tangent bundle of $\mathcal{T}_{g,n}$ along paths in $\mathcal{T}_{g,n}$, and taking the infimum of lengths over all piecewise $C^1$-paths.
Thurston \cite{Thurston} also gave an explicit formula for the weak norm of a tangent vector $V$ at a point $X$ in $\mathcal{T}_{g,n}$, namely,

\begin{equation}\label{eq:norm-L}
\Vert V \Vert_{\mathrm{Th}} =\sup_{\lambda \in \mathcal{ML}} \frac{d \ell_{\lambda}(V)}{\ell_{\lambda}(X)}.
\end{equation}
Here, $\mathcal{ML}$ is the space of measured laminations on the surface, $\ell_{\lambda}:\mathcal{T}_{g,n}\to \mathbb{R}$ is the
 geodesic length function on Teichm\"uller space associated to the measured lamination $\lambda$ and $d \ell_{\lambda}$ is the differential of  $\ell_{\lambda}$ at the point $X\in \mathcal{T}_{g,n}$.

There is a (non-necessary unique) extremal Lipschitz homeomorphism that realizes the infimum in \eqref{eq:L}.
 Related to the extremal Lipschitz homeomorphsim, there is a class of geodesics for Thurston's metric called \emph{stretch lines},
 which we will describe below.

Let $X$ be again a hyperbolic surface on $R$. A geodesic lamination $\lambda$ on $X$ is said to be \emph{complete} if its complementary regions are all isometric to ideal triangles.  Associated with $(X, \lambda)$ is a measured foliation $F_{\lambda}(X)$, called the horocyclic foliation, whose equivalence class is characterized by the following three properties:
\begin{enumerate}[(i)]
\item $F_{\lambda}(X)$ intersects $\lambda$ transversely, and in each cusp of an ideal triangle in the complement of $\lambda$,
the leaves of the foliation are pieces of horocycles that make right angles with the boundary of the triangle;
\item  on the leaves of $\lambda$, the transverse measure for $F_{\lambda}(X)$ agrees with hyperbolic arc length;
\item there is a non-foliated region at the center of each ideal triangle of $X\setminus \lambda$ whose boundary consists of three pieces of horocycles that are pairwise tangent (see Figure \ref{fig:horo}).
 \end{enumerate}

\begin{figure}[htbp]

\centering

\includegraphics[width=12cm]{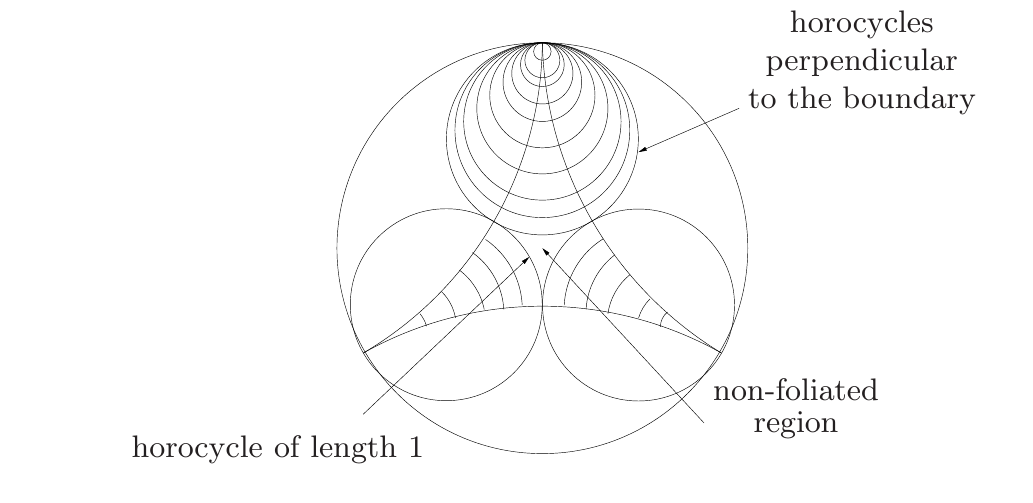}

\caption{The horocyclic foliation of an ideal triangle.}
\label{fig:horo}
\end{figure}

We denote by $\mathcal{MF}(\lambda)$ the space of measured foliations that are transverse to $\lambda$. Note that by the definition of a
horocylic foliation, we require the measured foliation in $\mathcal{MF}(\lambda)$ to be  standard in a neighborhood of any cusp of the
surface. This means that its
leaves are circles, and the transverse measure of any arc converging to the cusp is infinite. Thurston \cite{Thurston} proved the following fundamental result.
\begin{theorem}\label{th:T}
The map $\phi_\lambda:  \mathcal{T}_{g,n}\to \mathcal{MF}(\lambda)$ defined by $ X \mapsto F_\lambda(X)$ is a homeomorphism.
\end{theorem}

The  \emph{stretch line $g_\lambda^t(X)$ directed by $\lambda$ and passing through $X\in  \mathcal{T} (R)$} is the line 
in Teichm\"uller space parameterized by
$$\mathbb{R} \ni t \mapsto g_\lambda^t(X)= \phi_\lambda^{-1}(e^t F_\lambda(X)).$$

We call a segment of a stretch line \emph{a stretch path}. We also have a natural notion of \emph{stretch ray.}

Stretch rays are geodesics for the Thurston metric: Suppose that $\lambda_0$ is a measured lamination whose support
is contained in a complete geodesic lamination $\lambda$. Let $\Gamma(t)=g_\lambda^t(X)$.
Then, for any two points $\Gamma_s, \Gamma_t, s\leq t$ on the stretch line,
The  distance $d_{\mathrm{Th}}(\Gamma_s, \Gamma_t)$ is equal to $t-s$, and this distance is realized by
$$\log  \frac{\ell_{\lambda_0} (\Gamma_t)}{\ell_{\lambda_0} (\Gamma_s)}.$$
It was observed by Thurston \cite{Thurston} that any measured  lamination that realizes the maximum of
$$\sup_{\mu \in \mathcal{ML}}  \frac{\ell_\mu(X_t)}{\ell_\mu(X_s)}$$
is supported by $\lambda$. The union of all the measured geodesic
laminations that realize this maximum is also a measured geodesic lamination, called the \emph{stump}
of $\lambda$.

Thurston proved that any two points in Teichm\"uller space can be joined by a geodesic which is a finite concatenation of stretch paths,
but in general such a geodesic is not unique. There also exist geodesics
for Thurston's metric that are not concatenations of stretch paths. Some of them are made explicit in \cite{PT}. This contrasts
with Teichm\"uller's theorem establishing the existence and uniqueness of Teichm\"uller geodesics
joining any two distinct points.

Given $X\in \mathcal{T}_{g,n}$ and a complete geodesic lamination $\lambda$ on $X$, we consider the map
\begin{eqnarray*}
\Gamma(t) : \mathbb{R}_{\geq 0} &\to& \mathcal{T}_{g,n} \\
t &\mapsto& \Gamma_\lambda^t(X).
\end{eqnarray*}
By definition, $\Gamma(t)$ is the stretch ray directed by $\lambda$ starting at $X$.
Note that $\Gamma(0)=X$. There is a unique measured lamination $\mu$ which is equivalent to the
horocylic foliation $F_\lambda(X)$. In fact, there is a one-to-one correspondence between measured laminations and
(equivalence class of) measured foliations on $X$. The measured lamination $\mu$ equivalent to $F_\lambda(X)$
is totally transverse to $\lambda$ (see Thurston \cite[Proposition 9.4]{Thurston}).

In the following, we assume that $\lambda$ has no closed leaves.
It follows from this assumption that $\lambda$ is obtained from its stump by adding finitely many infinite geodesics.
Therefore any
simple closed geodesic or any geodesic arc (connecting two simple closed geodesics $\beta_1, \beta_2$
with $i(\beta_1,\beta_2)=0$ perpendicularly) is transverse to $\lambda$.  We shall use this fact. Papadopoulos \cite{Papa}
proved the following:
\begin{lemma}\label{lemma:papa}
 For any simple closed curve $\gamma$ on $R$, there is a constant $C_\gamma$ that depends only on $\gamma$ such that
$$e^ti(\mu,\gamma)\leq \ell_\gamma(\Gamma(t))\leq e^t i(\mu,\gamma)+C_\gamma.$$
\end{lemma}
This implies that,
as $t\to +\infty$, $\Gamma(t)$ converges to $[\mu]$, the projective class of $\mu$
on the boundary of Thurston's compactification.

When $i(\mu,\gamma)=0$, Lemma \ref{lemma:papa} says that $\ell_\gamma(\Gamma(t))$ is bounded above by
a constant $C_\gamma$ (depending on $\gamma$).  The following result of Th\'eret \cite{Theret}
gives a further estimate for $\ell_\gamma(\Gamma(t))$.

\begin{lemma}\label{lem:theret}
Let $\gamma$ be a simple closed curve on $R$ with $i(\mu,\gamma)=0$. If $\gamma$ is a leaf of $\mu$ with
wight equal to $\omega_\gamma$, then
$$\ell_\gamma(\Gamma(t))\leq \frac{3|\chi(R)|}{\sinh(e^t \omega_\gamma/2)}.$$
If $\gamma$ is not a leaf of $\mu$ (in this case we set $\omega_\gamma=0$), then
$$B_\gamma\leq \ell_\gamma(\Gamma(t))\leq C_\gamma,$$
where $B_\gamma$ and $C_\gamma$ are positive constants that depend only on $\gamma$.
\end{lemma}

\section{Geometry of the arc metric}\label{sec:arc}
In this section, we prove our main theorem.
We first recall the definition of the arc metric. Then we introduce the horofunction compactification of the arc metric.
Finally, we show that Thurston's compactification $\overline{\mathcal{T}(S)}$ is homeomorphic
to the horofunction compactification of the arc metric.

\subsection{The arc metric}

For any
$\gamma\in \mathcal{A}(S)\cup \mathcal{C}(S)$ and for any hyperbolic structure $X$ on $S$, we let
$\gamma^X$ be the
geodesic representative of $\gamma$
 (that is, the curve of shortest length in the homotopy
class relative to $\partial S$). In the case where
 $\gamma$ is an equivalence class of arcs, the geodesic
$\gamma^X$ is unique, and it is orthogonal to $\partial X$ at each
intersection point. We denote by
$\ell_\gamma(X)$
 the length of
$\gamma^X$ with respect to the hyperbolic metric considered. This length
depends only on the equivalence class of $X$ in Teichm\"uller space.

Let $S$ be a hyperbolic surface with geodesic boundary. Let $\mathcal{C}=\mathcal{C}(S)$ and $\mathcal{A}=\mathcal{A}(S)$.
In the paper \cite{LPST}, the authors defined an asymmetric metric,
 the \emph{arc metric}, on $\mathcal{T}(S)$ by \begin{equation} \label{eq:LL}d(X,Y)=\log \sup_{\gamma\in \mathcal{C}\cup \mathcal{A}}  \frac{\ell_\gamma(Y)}{\ell_\gamma(X)}.
\end{equation}
Relations between the arc metric and the Teichm\"uller metric are studied in the same paper.

\begin{remark}
Note that the arcs are necessary in order to have a metric because if we use only the closed curves, then there exist $X,Y$ such that (see \cite{PT10})
$$\log \sup_{\gamma\in \mathcal{C}}  \frac{\ell_\gamma(Y)}{\ell_\gamma(X)}<0.$$
\end{remark}

The definition of the arc metric is a natural generalization of Thurston's formula $(\ref{equ:Thurston})$.
\begin{proposition}[\cite{LPST}]
The  map $\Psi$ (defined in Section \ref{sec:pre}) gives an isometric embedding
\begin{equation*}\label{eq:doubling}
\left( \mathcal{T}(S), d \right) \hookrightarrow \left(\mathcal{T}(S^d),d_{\mathrm{Th}} \right),
\end{equation*}
that is,

$$d(X,Y)=d_{\mathrm{Th}}(X^d,Y^d).$$
\end{proposition}

\subsection{Horofunction compactification}

Let $\mathcal{T}(S)$ be the Teichm\"uller space of $S$
endowed with the arc metric $d$. We set $\bar d(X,Y)=d(Y,X)$. Then, $\bar d$ is also an asymmetric metric on  $\mathcal{T}(S)$.
The topology of $\mathcal{T}(S)$ induced by the arc metric $d$ is the
same as the one induced by $\bar d$, and it is defined as the topology induced by the genuine metric  $d+\bar d$ or $\delta=\max\{d,\bar d\}$ (see  \cite[Theorem 4.4]{LPST}).

 Fix a base point $X_0\in \mathcal{T}(S)$.
To each $X\in \mathcal{T}(S)$ we assign a function
$\Phi_X: \mathcal{T}(S) \to \mathbb{R}$, defined by
\begin{equation*}\label{Equation:Psi}
\Phi_X(Y)=d(Y,X)-d(X_0,X).
\end{equation*}
Let $C(\mathcal{T}(S))$ be the space of continuous functions on $\mathcal{T}(S)$ endowed with the topology
of locally uniform convergence.
Then the map
\begin{eqnarray*}
\Phi:\mathcal{T}(S)&\to& C(\mathcal{T}(S)), \\
 X&\mapsto& \Phi_X
\end{eqnarray*} is an embedding.
The closure $\overline{\Phi(\mathcal{T}(S))}$ is compact (this follows from the fact that
$\mathcal{T}(S)$ is locally compact and the Arzel\'a-Ascoli theorem) and it is called the \emph{horofunction compactification}  \index{Horofunction compactification}  of $\mathcal{T}(S)$. The \emph{horofunction boundary} is defined to be
$$\overline{\Phi(\mathcal{T}(S))}-\Phi(\mathcal{T}(S)),$$
and its elements are called \emph{horofunctions}.

\begin{remark}
For a general locally compact metric space $(M,d)$, the horofunction compactification is defined
by Gromov \cite{Gromov}. A good property of the horofunction compactification is that the action of the isometry group $\mathrm{Isom} (M, d)$  of $M$
extends continuously to a homeomorphism on the horofunction
boundary.
\end{remark}

Note that our definition  depends on the choice of a base point $X_0$. However, if we let
 $$\widetilde{\Phi}_X=d(\cdot,X)-d(Y_0,X)$$
 for another base point $Y_0$, then the relation between $\Phi_X$ and $\widetilde{\Phi}_X$
 is described by
\begin{equation}\label{equ:basis}
\widetilde{\Phi}_X(\cdot)=\Phi_X(\cdot)-\Phi_X(Y_0).
\end{equation}
Equation $(\ref{equ:basis})$ induces a natural homeomorphism between $\Psi(\mathcal{T}(S))$
and $\widetilde{\Psi}(\mathcal{T}(S))$ and it induces a homeomorphism between
the corresponding horofunction boundaries.
As a result, we can embed  the Teichm\"uller space $\mathcal{T}(S)$ into the quotient of $C(\mathcal{T}(S))$
by the $1$-dimensional subspace of constant functions, by identifying two functions in $C(\mathcal{T}(S))$ whenever they differ by an additive constant. For convenience, in the following discussion, we shall fix a base point.

In the remaining part of this paper, we shall make the identification
$$\mathcal{PML}\cong \{\eta\in \mathcal{ML}(S) \ | \ \ell_\eta(X_0)=1\}.$$

Suppose that $X\in \mathcal{T}(S)$. From the definition,

$$\Phi_X(\cdot)=\log \sup_{\eta\in \mathcal{PML}}  \frac{\ell_\eta(X)}{\ell_\eta(\cdot)}-
\log \sup_{\eta\in \mathcal{PML}}  \frac{\ell_\eta(X)}{\ell_\eta(X_0)}.$$

For any $\gamma\in \mathcal{ML}$, we set
$$\mathcal{L}_\gamma(X)=\ell_\gamma(X)/\sup_{\eta\in \mathcal{PML}}  \frac{\ell_\eta(X)}{\ell_\eta(X_0)}.$$

Then
\begin{equation}\label{eq:Q}
\Phi_X(\cdot)=\log \sup_{\gamma\in \mathcal{PML}} \frac{\mathcal{L}_\gamma(X)}{\ell_\gamma(\cdot)}.
\end{equation}

\subsection{Convergence in Thurston's compactification}
Let $(X_n)$ be a sequence in $\mathcal{T}(S)$  that converges to $\mu\in \mathcal{PML}$.
From the definition,
there exists a sequence of numbers $(c_n), c_n>0$,
such that for any $\gamma\in \mathcal{ML}$, $c_n\ell_\gamma(X_n)\to i(\mu,\gamma)$ as $n\to\infty$.
We claim that the following holds:

\begin{lemma}
With the above notation, we have:
\[\mathcal{L}_\gamma(X_n)\to i(\mu,\gamma)/\sup_{\nu\in \mathcal{PML}}  \frac{i(\mu,\nu)}{\ell_\nu(X_0)}\ \hbox{as} \ n\to\infty.\]
\end{lemma}

\begin{proof}
Note that

\begin{eqnarray*}
\mathcal{L}_\gamma(X_n)&=& \ell_\gamma(X_n)/ \sup_{\eta\in \mathcal{PML}}\frac{\ell_\eta(X_n)}{\ell_\eta(X_0)}\\
&=& c_n\ell_\gamma(X_n)/ \sup_{\eta\in \mathcal{PML}}\frac{c_n\ell_\eta(X_n)}{\ell_\eta(X_0)}\\
\end{eqnarray*}

By assumption, $c_n\ell_\eta(X_n)\to i(\mu,\eta)$ (as $n\to\infty$) for all $\eta\in \mathcal{PML}$.
By a continuity argument (the same proof as \cite[Lemma 3.1]{Walsh}),
we have $c_n\ell_\eta(X_n)\to i(\mu,\eta)$ uniformly on $\mathcal{PML}$.
This implies that
$$\lim_{n\to\infty}\sup_{\eta\in \mathcal{PML}}\frac{c_n\ell_\eta(X_n)}{\ell_\eta(X_0)}=\sup_{\eta\in \mathcal{PML}}\frac{i(\mu,\eta)}{\ell_\eta(X_0)}.$$

Since $c_n\ell_\gamma(X_n)\to i(\mu,\gamma)$ as $n\to\infty$, we are done.
\end{proof}

For $\gamma$ and $\mu$ in $\mathcal{ML}$, we set
$$\mathcal{L}_\gamma(\mu)=i(\mu,\gamma)/\sup_{\nu\in \mathcal{PML}}  \frac{i(\mu,\nu)}{\ell_\nu(X_0)}.$$
Note that the value $\mathcal{L}_\gamma(\mu)$ is invariant by multiplication of $\mu$ by a positive constant, therefore we can also define $\mathcal{L}_\gamma(\mu)$ by the same formula for $\mu$ in $\mathcal{PML}$.

\begin{proposition}\label{pro:convergent}
A sequence $(X_n)$ in $\mathcal{T}(S)$ converges to $\mu\in \mathcal{PML}$
if and only if $\mathcal{L}_\gamma(X_n)$ converges to $\mathcal{L}_\gamma(\mu)$ for all
$\gamma\in \mathcal{ML}$.
\end{proposition}
\begin{proof}
We already showed that if  $(X_n)$ converges to $\mu$, then
$\mathcal{L}_\gamma(X_n)$ converges to $\mathcal{L}_\gamma(\mu)$ for all
$\gamma\in \mathcal{ML}$.

Conversely, assume that $\mathcal{L}_\gamma(X_n)$ converges to $\mathcal{L}_\gamma(\mu)$ for all
$\gamma\in \mathcal{ML}$. Then $(X_n)$ is unbounded in $\mathcal{T}(S)$.
Let $(Y_n)$ be any subsequence
of $X_n$ that converges to $\mu'\in \mathcal{PML}$. Then
$\mathcal{L}_\gamma(Y_n)$ converges to $\mathcal{L}_\gamma(\mu')$ for all
$\gamma\in \mathcal{ML}$. By assumption,
$\mathcal{L}_\gamma(\mu')=\mathcal{L}_\gamma(\mu)$, therefore
$$i(\mu,\gamma)/\sup_{\nu\in \mathcal{PML}}  \frac{i(\mu,\nu)}{\ell_\nu(X_0)}=
i(\mu',\gamma)/\sup_{\nu\in \mathcal{PML}}  \frac{i(\mu',\nu)}{\ell_\nu(X_0)}.$$

Therefore, if we set
$$C=\sup_{\nu\in \mathcal{PML}}  \frac{i(\mu,\nu)}{\ell_\nu(X_0)}/ \sup_{\nu\in \mathcal{PML}}  \frac{i(\mu',\nu)}{\ell_\nu(X_0)},$$
then $i(\mu,\gamma)=Ci(\mu',\gamma)$
for all $\gamma\in \mathcal{ML}$. This implies that $\mu=\mu'$ in $\mathcal{PML}$.
Since $(Y_n)$ is arbitrary, $(X_n)$ converges to $\mu$.
\end{proof}

\begin{corollary}\label{coro:convergent}
A sequence $(Z_n)$ in $\overline{\mathcal{T}(S)}$ converges to $Z\in \overline{\mathcal{T}(S)}$
if and only if $\mathcal{L}_\gamma(Z_n)$ converges to $\mathcal{L}_\gamma(Z)$ for all
$\gamma\in \mathcal{ML}$.
\end{corollary}
\begin{proof}
This follows from Proposition \ref{pro:convergent} and a usual continuity argument.
\end{proof}

For $\mu\in \mathcal{PML}(S)$, let $\Phi: \mathcal{PML} \to \mathcal{C}(\mathcal{T}(S))$ be the function defined by
\begin{equation}\label{eq:f} \Phi_\mu(\cdot)=\log \sup_{\gamma\in \mathcal{PML}} \frac{\mathcal{L}_\gamma(\mu)}{\ell_\gamma(\cdot)}.
\end{equation}

The maps on $\mathcal{PML}$ defined by Equations (\ref{eq:Q}) and (\ref{eq:f}) combine together and define a map

\begin{eqnarray*}
\Phi:\overline{\mathcal{T}(S)}&\to& C(\mathcal{T}(S)), \\
 Z&\mapsto& \Phi_Z.
\end{eqnarray*}

 By Corollary \ref{coro:convergent} and the compactness of $\mathcal{PML}$,
this map is continuous.
In \S \ref{sec:main}, we will prove that $\Phi$ is injective on $\overline{\mathcal{T}(S)}$.
The same result for surface without boundary was proved by Walsh \cite{Walsh}
by a direct method. Unfortunately, his argument does not apply here.
Our proof is based on the inequality (Lemma \ref{lemma:key}) in next section.

 \bigskip

\section{An inequality for length functions}\label{sec:inequality}
For any $\mu\in \mathcal{PML}$,
let  $\mu^d$ be the double of $\mu$ on $S^d$.
We endow $S^d$ with the hyperbolic structure $X_0^d$ and we choose a complete geodesic lamination $\lambda$ which contains no closed leaves and which is totally transverse to $\mu^d$.
(Recall that this is equivalent to saying that $\mu^d$ can be represented by a measured foliation transverse to $\lambda$ and trivial around each puncture.)

Denote by $\Gamma(t)$ the stretch line in $\mathcal{T}(S^d)$ directed by $\lambda$ and converging to $\mu^d$
in the positive direction, that is,

$$\Gamma(t)= \phi_\lambda^{-1}(e^t \mu^d)$$
where $\phi_\lambda$ is the map in Theorem \ref{th:T}.  For $t\geq 0$, the hyperbolic structure $\Gamma(t)$ might not be symmetric, and this is the reason for the technical work that follows.

Consider any $\alpha\in \mathcal{A}$.
We realize  $\alpha$ as a geodesic arc $\alpha_t$ on $\Gamma(t)$ whose endpoints are on two simple closed geodesics $\beta_1, \beta_2$
and which meets them perpendicularly. These closed geodesics are homotopic to the images in the hyperbolic surface $\Gamma(t)$ of the boundary curves of $S$ which contain the endpoints of $\alpha$. They can either coincide in $\Gamma(t)$ or be distinct, depending on whether they come from curves that coincide or are distinct in $S$.

Similarly, we can realize $\mu$ as a measured geodesic lamination $\mu_t$ on $\Gamma(t)$.
The support of $\mu_t$ lies on a totally geodesic subsurface of $\Gamma(t)$ which is
homeomorphic to $S$.
 The intersection number $i(\mu,\alpha)$ is realized by the total mass
 of the intersection of $\alpha_t$ with
$\mu_t$. Thus, we have:
$$i(\mu,\alpha)=  I (\mu_t, \alpha_t)$$
where

$$I(\mu_t, \alpha_t)=\int_{\alpha_t} d \mu_t.$$

We wish to prove an inequality  similar to \cite[Lemma 4.9]{Papa}.
The first step is to show that
there is a constant $C>0$ (depending only on the stretch line) such that for all $\alpha\in \mathcal{A}$,
$$e^ti(\mu,\alpha)-C \leq \ell_\alpha(\Gamma(t)).$$
This is confirmed by Lemma \ref{lem:lower} below.

We fix  $\alpha$ in $\mathcal{A}$ and the hyperbolic structure $\Gamma(t)$.
We will use the same notation $\alpha$ to denote the geodesic representation of $\alpha$ on $\Gamma(t)$.
We suppose that $\alpha$ joins two simple closed geodesics $\beta_1, \beta_2$
perpendicularly.  We set $\ell(\alpha)=\ell_\alpha(\Gamma(t))$ and so on.

\begin{remark}
It seems that the constant $C>0$ is necessary when $\alpha\in \mathcal{A}$.
This is due to the fact that the horocylic foliation $F_t$ equivalent to $e^t\mu^d$ is not symmetric. A similar argument as in \cite[Lemma 4.9]{Papa} shows that for any $\alpha\in \mathcal{C}$,
$e^ti(\mu^d, \alpha)\leq \ell_\alpha(\Gamma(t)).$
This can be done by showing that
$\ell_\alpha(\Gamma(t))\geq I(F_t, \alpha_t),$
where $\alpha_t$ is the geodesic representation of $\alpha$ on $\Gamma(t)$.
\end{remark}
\subsection{Estimation of arc length in a pair of pants}
The three geodesics $\beta_1,\beta_2, \alpha$ determine a geodesic pair of pants,  denoted by $\mathcal{P}$, which is isotopic to a  tubular neighborhood of $\alpha \cup \beta_1 \cup \beta_2$.

When $\beta_1 = \beta_2$ (and in this case we denote both curves by $\beta$), the boundary of $\mathcal{P}$ has three connected components:
one is $\beta$ and the other two will be denoted by $\gamma_1, \gamma_2$.
It may happen that $\gamma_1$ and $\gamma_2$  coincide on the surface $S^d$.

If $\beta_1 \neq \beta_2$, the boundary of $\mathcal{P}$ has three connected components, two of them are $\beta_1$ and $\beta_2$.
We denote by $\gamma$ the third one, so that $\partial \mathcal{P} = \beta_1 \cup \beta_2 \cup \gamma$.

\begin{remark}
In both cases, some boundary component of $\mathcal{P}$ (such as $\gamma, \gamma_1$ or $\gamma_2$) may have  zero length. We always consider a puncture to be
a boundary component of length zero.
\end{remark}

The intersection numbers of the three  boundary components of $\mathcal{P}$ with $\mu$ are three positive numbers
satisfying some equation.
To simplify notation,
we will always assume that indices are chosen such that $i(\mu,\gamma_1) \geq i(\mu,\gamma_2)$ in the first case,
and that $i(\mu,\beta_1) \geq i(\mu,\beta_2)$ in the second case.

\begin{figure}[htbp]

\centering

\includegraphics[width=12cm]{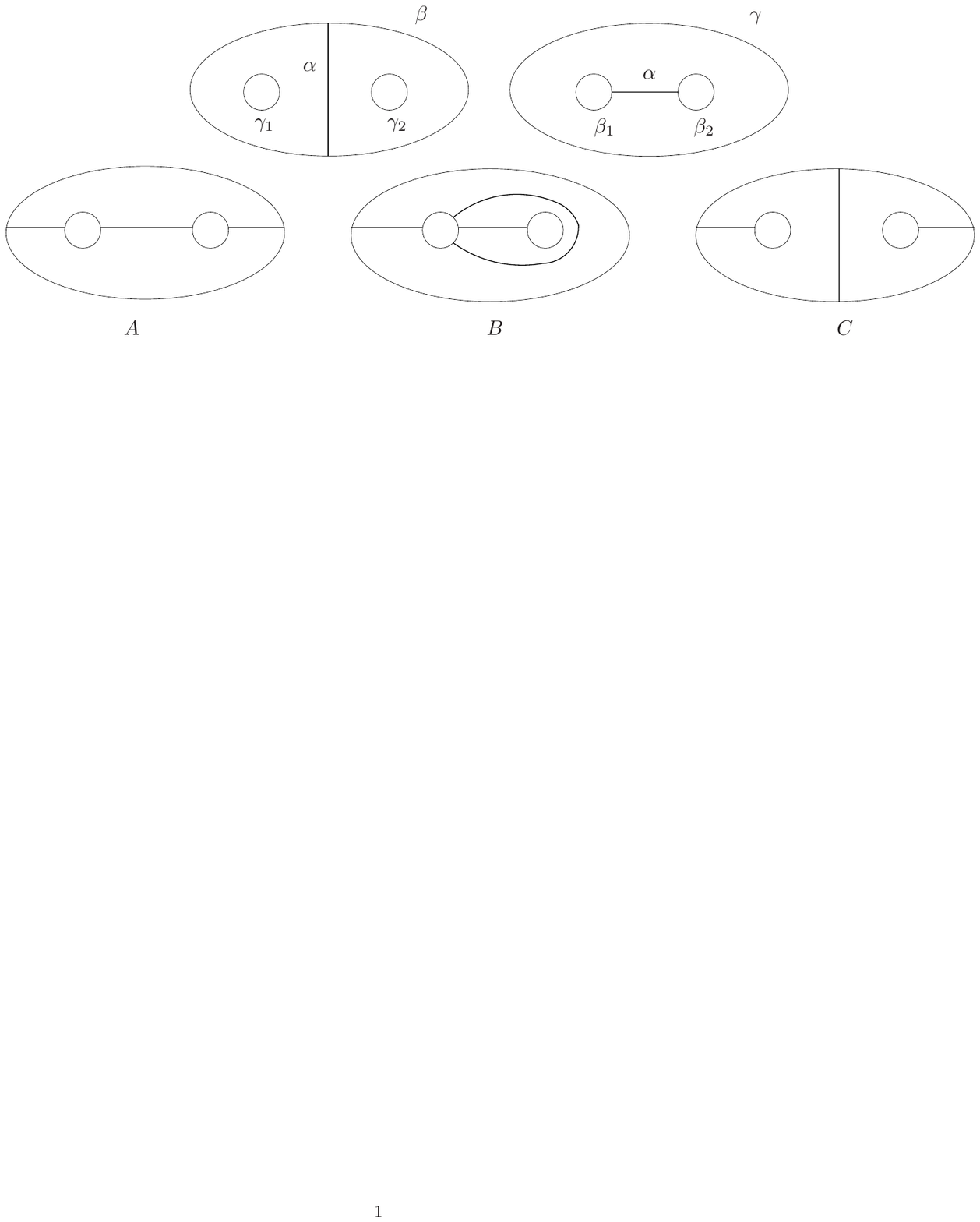}
\caption{The pair of pants containing the arc $\alpha$ falls into two types. For each type, there
are three cases illustrated in (A)-(C).}
\label{fig:pants}
\end{figure}

As indicated on the left of Figure \ref{fig:pants}, the case where $\beta_1 = \beta_2$
is divided into three different subcases.

 \begin{enumerate}[(A)]
 \item the intersection number of $\mu$ with one boundary component of $\mathcal{P}$ is less than the sum of
 the intersection number of $\mu$ with the two others. (That is, the triangle inequality for the triple of intersection numbers holds.)
 \item $i(\mu,\gamma_1)> i(\mu,\beta)+i(\mu,\gamma_2)$.
 \item $i(\mu,\beta)>i(\mu,\gamma_1)+i(\mu,\gamma_2)$.
 \end{enumerate}

In each subcase, we have the following corresponding equation:

 \begin{enumerate}[(A)]
 \item $i(\mu,\alpha)=\frac{1}{2}\left( i(\mu,\gamma_1)+i(\mu,\gamma_2)-i(\mu,\beta)\right)+\omega_\beta$
 \item $i(\mu,\alpha)=i(\mu,\gamma_1)-i(\mu,\beta)+\omega_\beta$.
 \item $i(\mu,\alpha)=0$.
 \end{enumerate}
Here $\omega_\beta$ is the weight of $\beta$ in $\mu^d$.  We clearly have:
$$i(\mu^d,\gamma_1)=i(\mu,\gamma_1), i(\mu^d,\gamma_2)=i(\mu,\gamma_2), i(\mu^d, \beta)=2i(\mu,\beta).$$

Now we give a lower bound of $\ell(\alpha)$ in terms of $\ell(\beta),\ell(\gamma_1),\ell(\gamma_2)$ for all
cases (A)-(C).
We need the following formula, which can be shown by combining the hyperbolic pentagon and hexagon formulae, 
see Buser \cite[\S 2.4]{Buser}.

\begin{align} \label{formula:case1}
\cosh^2\left(\tfrac{1}{2}\ell(\alpha)\right) &= \frac{-1 + \cosh^2\left(\tfrac{1}{2}\ell(\beta)\right) + \cosh^2\left(\tfrac{1}{2}\ell(\gamma_1)\right) + \cosh^2\left(\tfrac{1}{2}\ell(\gamma_2)\right)}{\sinh^2\left(\tfrac{1}{2}\ell(\beta)\right)} + \\
&+ \frac{2 \cosh\left(\tfrac{1}{2}\ell(\beta)\right) \cosh\left(\tfrac{1}{2}\ell(\gamma_1)\right) \cosh\left(\tfrac{1}{2}\ell(\gamma_2)\right)}{\sinh^2\left(\tfrac{1}{2}\ell(\beta)\right)}.   \nonumber
\end{align}

We also need some elementary estimates:

\begin{enumerate}[(i)]
\item For $x\geq 0$, $\frac{1}{2}e^x\leq \cosh(x)\leq e^x ; \frac{1}{4}e^{2x}\leq \cosh^2(x)\leq e^{2x}$.
\item If $x>A>0$, then $$\frac{1}{2}(1-e^{2A})e^x\leq \sinh(x) \leq \frac{1}{2}e^x;$$
 if $0<x<1$, then
$$x<\sinh(x)<2x.$$
\item For each $\gamma\in \mathcal{C}(S^d)$, we have (recalling that we denote by $\ell(\gamma)$ the geodesic
length of $\gamma$ on $\Gamma(t)$)
$$\frac{1}{2}\exp(\frac{1}{2}e^ti(\mu^d,\gamma))\leq \cosh(\frac{1}{2}\ell(\gamma))\leq \exp(\frac{1}{2}\ell(\gamma))\leq \exp(\frac{C_\gamma}{2})\exp(\frac{1}{2}e^ti(\mu^d,\gamma)).$$
\item If $\gamma\in \mathcal{C}(S^d)$ is a leaf of $\mu^d$ (that is, $\omega_\gamma>0$), then
$$\ell(\gamma)\leq 12|\chi(S^d)|\exp(-\frac{1}{2}\omega_\gamma e^t).$$
\end{enumerate}

The inequality in (iii) follows from Lemma \ref{lemma:papa}.
The inequality in (iv) follows from Lemma \ref{lem:theret} and the fact that $1/\sinh(x)\leq 4/e^x$ for $ x>0$.

\bigskip

\textbf{Case (A).} We rewrite Formula \eqref{formula:case1} in the following way:
\begin{equation*}
\cosh^2\left(\frac{1}{2}\ell(\alpha)\right)= 2\coth\left(\frac{1}{2}\ell(\beta)\right) \frac{\cosh\left(\frac{1}{2}\ell(\gamma_1)\right)\cosh\left(\frac{1}{2}\ell(\gamma_2)\right)}{\sinh\left(\frac{1}{2}\ell(\beta)\right)}(1+R_A)
\end{equation*}
where the term $R_A$ is given by

\begin{align*}
\frac{-1 + \cosh^2\left(\tfrac{1}{2}\ell(\beta)\right) + \cosh^2\left(\tfrac{1}{2}\ell(\gamma_1)\right) + \cosh^2\left(\tfrac{1}{2}\ell(\gamma_2)\right)}{2 \cosh\left(\tfrac{1}{2}\ell(\beta)\right) \cosh\left(\tfrac{1}{2}\ell(\gamma_1)\right) \cosh\left(\tfrac{1}{2}\ell(\gamma_2)\right)}
\end{align*}
It is easy to see that that $R_A>0$.

If $\omega_\beta=0$, then $\ell(\beta)\geq B_\beta$ (see lemma \ref{lem:theret}), and we have
$$1<  \coth\left(\tfrac{1}{2}\ell(\beta)\right) \leq  \coth\left(\tfrac{1}{2}B_\beta\right).$$

To give a lower bound for $\ell(\alpha)$, note that

\begin{eqnarray*}
\exp\left(\ell(\alpha)\right) &\geq&  \cosh^2\left(\tfrac{1}{2}\ell(\alpha)\right) \\
&\geq& 2\coth\left(\frac{1}{2}\ell(\beta)\right) \frac{\cosh\left(\frac{1}{2}\ell(\gamma_1)\right)\cosh\left(\frac{1}{2}\ell(\gamma_2)\right)}{\sinh\left(\frac{1}{2}\ell(\beta)\right)}\\
&\geq& 2 \frac{\cosh\left(\frac{1}{2}\ell(\gamma_1)\right)\cosh\left(\frac{1}{2}\ell(\gamma_2)\right)}{\sinh\left(\frac{1}{2}\ell(\beta)\right)} \\
&\geq& \frac{e^{\frac{1}{2}\ell(\gamma_1)}e^{\frac{1}{2}\ell(\gamma_2)}}{e^{\frac{1}{2}\ell(\beta)}} = \exp\left(\frac{\ell(\gamma_1)}{2}+\frac{\ell(\gamma_2)}{2}-\frac{\ell(\beta)}{2}\right).
\end{eqnarray*}
By taking the logarithm of each side and applying Lemma \ref{lemma:papa}, we have

\begin{eqnarray*}
\ell(\alpha) &\geq& \frac{\ell(\gamma_1)}{2}+\frac{\ell(\gamma_2)}{2}-\frac{\ell(\beta)}{2}\\
&\geq& e^t \frac{1}{2}\left( i(\mu,\gamma_1)+i(\mu,\gamma_2)-i(\mu,\beta)\right)-C_\beta \\
&=& e^t i(\mu,\alpha)- C_\beta.
\end{eqnarray*}

If $\omega_\beta>0$, we have $i(\mu,\beta)=0$. Moreover, the length $\ell(\beta)$ is less than $C_\beta$
and it is less than $1$ when $t$ is sufficiently large. As a result, we may assume
(using the second inequality in (ii)) that
$$\sinh\left(\frac{1}{2}\ell(\beta)\right)\leq \ell(\beta).$$
Then we have

\begin{eqnarray*}
\exp\left(\ell(\alpha)\right) &\geq&  \cosh^2\left(\tfrac{1}{2}\ell(\alpha)\right) \\
&\geq& 2\coth\left(\frac{1}{2}\ell(\beta)\right) \frac{\cosh\left(\frac{1}{2}\ell(\gamma_1)\right)\cosh\left(\frac{1}{2}\ell(\gamma_2)\right)}{\sinh\left(\frac{1}{2}\ell(\beta)\right)}\\
&\geq&\frac{\cosh\left(\frac{1}{2}\ell(\gamma_1)\right)\cosh\left(\frac{1}{2}\ell(\gamma_2)\right)}{\sinh^2\left(\frac{1}{2}\ell(\beta)\right)}\\
&\geq& \frac{\frac{1}{2}\exp\left( \frac{1}{2}\ell(\gamma_1)+ \frac{1}{2}\ell(\gamma_2)\right)}{\ell(\beta)^2}.
\end{eqnarray*}
Applying (iv), we get
\begin{eqnarray*}
\exp\left(\ell(\alpha)\right) &\geq&  \frac{\frac{1}{2}\exp\left( \frac{1}{2}\ell(\gamma_1)+ \frac{1}{2}\ell(\gamma_2)\right)}{(12|\chi(S^d)|)^2 \exp\left(-\omega_\gamma e^t \right)}.
\end{eqnarray*}
Taking the logarithm of each side, we have
\begin{eqnarray*}
\ell(\alpha) &\geq&   \frac{1}{2}\ell(\gamma_1)+ \frac{1}{2}\ell(\gamma_2)+\omega_\gamma e^t -\log(288|\chi(S^d)|^2) \\
&\geq& \frac{1}{2} \left( i(\mu,\gamma_1)+i(\mu,\gamma_2)\right)e^t+\omega_\gamma e^t -\log(288|\chi(S^d)|^2)  \\
&=& e^t i(\mu,\alpha)- \log(288|\chi(S^d)|^2).
\end{eqnarray*}

\bigskip

\textbf{Case (B).} We can rewrite formula \ref{formula:case1} in the following way:

\begin{equation*}
\cosh^2\left(\tfrac{1}{2}\ell(\alpha)\right) =  \frac{\cosh^2\left(\tfrac{1}{2}\ell(\gamma_1)\right)}{\sinh^2\left(\tfrac{1}{2}\ell(\beta)\right)}   \left( 1 +  R_B\right)
\end{equation*}
where the term $R_B>0$ is given by

\begin{equation*}
\frac{-1 + \cosh^2\left(\tfrac{1}{2}\ell(\beta)\right) + \cosh^2\left(\tfrac{1}{2}\ell(\gamma_2)\right) + 2 \cosh\left(\tfrac{1}{2}\ell(\beta)\right) \cosh\left(\tfrac{1}{2}\ell(\gamma_1)\right) \cosh\left(\tfrac{1}{2}\ell(\gamma_2)\right)}{ \cosh^2\left(\tfrac{1}{2}\ell(\gamma_1)\right) }
\end{equation*}

Now, if $w_\beta = 0$, we have

\begin{eqnarray*}
\exp(\ell(\alpha))&\geq& \cosh^2\left(\tfrac{1}{2}\ell(\alpha)\right) \\
&\geq& \frac{\cosh^2\left(\tfrac{1}{2}\ell(\gamma_1)\right)}{\sinh^2\left(\tfrac{1}{2}\ell(\beta)\right)} \\
&\geq& \exp\left(\ell(\gamma_1)-\ell(\beta)\right).
\end{eqnarray*}
It follows that

\begin{eqnarray*}
\ell(\alpha)&\geq& \ell(\gamma_1)-\ell(\beta) \\
&\geq& e^t(i(\mu,\gamma_1)-i(\mu,\beta))-C_\beta \\
&=& e^ti(\mu,\alpha)-C_\beta.
\end{eqnarray*}

If, instead, $w_\beta > 0$, we have $i(\mu,\beta)= 0$ and $\ell(\beta)$ converges  to zero as $t$ tends to infinity.
Applying (iv), we have (for $t$ sufficiently large)

\begin{eqnarray*}
\exp(\ell(\alpha))&\geq& \frac{\cosh^2\left(\tfrac{1}{2}\ell(\gamma_1)\right)}{\sinh^2\left(\tfrac{1}{2}\ell(\beta)\right)}  \\
&\geq& \frac{\frac{1}{4}\exp\left(\ell(\gamma_1)\right)}{\left(\ell(\beta)\right)^2} \\
&\geq& \frac{1}{576|\chi(S^d)|^2}\exp\left(\ell(\gamma_1)+\omega_\beta e^t\right).
\end{eqnarray*}
Thus
\begin{eqnarray*}
\ell(\alpha)&\geq& \ell(\gamma_1)+\omega_\gamma e^t-\log \left(576|\chi(S^d)|^2\right)  \\
&\geq& e^t\left(i(\mu,\gamma_1)+\omega_\beta \right)- \log \left(576|\chi(S^d)|^2\right)\\
&=& e^t i(\mu,\alpha)- \log \left(576|\chi(S^d)|^2\right).
\end{eqnarray*}

\bigskip

\textbf{Case (C).} Since $i(\mu,\alpha)=0$, the inequality $\ell(\alpha) \geq i(\mu,\alpha)$ is trivial.

\bigskip

Now we consider the case where $\beta_1\neq \beta_2$. As we did before, we separate
the intersection pattern into three different cases:

 \begin{enumerate}[(A')]
 \item the intersection number of $\mu$ with one boundary component of $\mathcal{P}$ is less than the sum of
 the intersection number of $\mu$ with the two others.
 \item $i(\mu,\beta_1)> i(\mu,\beta_2)+i(\mu,\gamma)$.
 \item $i(\mu,\gamma)>i(\mu,\beta_1)+i(\mu,\beta_2)$.
 \end{enumerate}

Each of the above cases corresponds respectively to

 \begin{enumerate}[(A')]
 \item $i(\mu,\alpha)=\frac{1}{2}\left(\omega_{\beta_1}+\omega_{\beta_2}\right)$.
 \item $i(\mu,\alpha)=\frac{1}{2}\left(\omega_{\beta_1}+\omega_{\beta_2}\right)$.
 \item $i(\mu,\alpha)=\frac{1}{2}\left(i(\mu,\gamma)-i(\mu,\beta_1)-i(\mu,\beta_2)\right)+\omega_{\beta_1}+\omega_{\beta_2}$.
 \end{enumerate}

Recall the following formula:

\begin{equation} \label{formula:case2}
\cosh\left(\ell(\alpha)\right) = \frac{\cosh\left(\tfrac{1}{2}\ell(\gamma)\right) + \cosh\left(\tfrac{1}{2}\ell(\beta_1)\right) \cosh\left(\tfrac{1}{2}\ell(\beta_2)\right)}{\sinh\left(\tfrac{1}{2}\ell(\beta_1)\right)\sinh\left(\tfrac{1}{2}\ell(\beta_2)\right)}.
\end{equation}

\bigskip

\textbf{Case (A') or (B').} We can rewrite formula \eqref{formula:case2} in the following way:

\begin{equation}
\cosh\left(\ell(\alpha)\right) = \coth\left(\tfrac{1}{2}\ell(\beta_1)\right) \coth\left(\tfrac{1}{2}\ell(\beta_2)\right) \left(1 + S_{A,B}\right)
\end{equation}

where the term $S_{A,B}>0$ is given by

\begin{equation*}
\frac{\cosh\left(\tfrac{1}{2}\ell(\gamma)\right)}{\cosh\left(\tfrac{1}{2}\ell(\beta_1)\right) \cosh\left(\tfrac{1}{2}\ell(\beta_2)\right)}.
\end{equation*}

Now, if $w_{\beta_1} = w_{\beta_2} = 0$, we have $i(\mu, \alpha)=0$. Then it is obvious that $\ell(\alpha)\geq i(\mu,\alpha)$.

If $w_{\beta_1} = 0$ and $w_{\beta_2} > 0$, we have  $1 < \coth(\tfrac{1}{2}\ell(\beta_1)) \leq   \coth(\tfrac{1}{2}B_{\beta_1})$;
 while $\ell(\beta_2)$ goes to zero as $t$ tends to infinity.
  We may assume (by considering $t$ sufficiently large) that

  $$\coth(\tfrac{1}{2}\ell(\beta_2)) \geq \frac{1}{\sinh(\tfrac{1}{2}\ell(\beta_2))} \geq \frac{1}{\ell(\beta_2)}.$$
 Applying (iv), we have

\begin{eqnarray*}
\exp\left(\ell(\alpha)\right)&\geq& \cosh\left(\ell(\alpha)\right) \\
&\geq& \frac{1}{\ell(\beta_2)} \\
&\geq& \frac{1}{12|\chi(S^d)|}\exp\left(\frac{1}{2}\omega_{\beta_2} e^t\right).
\end{eqnarray*}
Thus
\begin{eqnarray*}
\ell(\alpha)&\geq& \frac{1}{2}\omega_{\beta_2} e^t-\log \left(12|\chi(S^d)|\right) \\
&=& e^ti(\mu,\alpha)-\log \left(12|\chi(S^d)|\right).
\end{eqnarray*}
The above argument applies also to the case where $w_{\beta_1} > 0$ and $w_{\beta_2} = 0$.

Now if $w_{\beta_1} > 0$ and $w_{\beta_2} > 0$, we have (for $t$ sufficiently large)

  $$\coth(\tfrac{1}{2}\ell(\beta_i)) \geq \frac{1}{\sinh(\tfrac{1}{2}\ell(\beta_i))} \geq \frac{1}{\ell(\beta_i)}, \ i=1,2.$$
 Applying (iv) again, we have

\begin{eqnarray*}
\exp\left(\ell(\alpha)\right)&\geq& \cosh\left(\ell(\alpha)\right) \\
&\geq& \frac{1}{\ell(\beta_1)\ell(\beta_2)} \\
&\geq& \frac{1}{144|\chi(S^d)|^2}\exp\left(\frac{1}{2}\omega_{\beta_1} e^t+\frac{1}{2}\omega_{\beta_2} e^t\right).
\end{eqnarray*}
Thus
\begin{eqnarray*}
\ell(\alpha)&\geq& \frac{1}{2}\left(\omega_{\beta_1}+\omega_{\beta_2} \right)e^t-\log \left(144|\chi(S^d)|^2\right) \\
&=& e^ti(\mu,\alpha)-\log \left(144|\chi(S^d)|^2\right).
\end{eqnarray*}

\textbf{Case (C').} we can rewrite formula \ref{formula:case2} in the following way:

\begin{equation*}
\cosh\left(\ell(\alpha)\right) = \frac{\cosh\left(\tfrac{1}{2}\ell(\gamma)\right)}{\sinh\left(\tfrac{1}{2}\ell(\beta_1)\right)\sinh\left(\tfrac{1}{2}\ell(\beta_2)\right)}  \left(1+ S_C\right)
\end{equation*}

where the term $S_C>0$ is given by

\begin{equation*}
\frac{\cosh\left(\tfrac{1}{2}\ell(\beta_1)\right) \cosh\left(\tfrac{1}{2}\ell(\beta_2)\right)}{\cosh\left(\tfrac{1}{2}\ell(\gamma)\right)}.
\end{equation*}

In this case, we have a lower bound for $\ell(\alpha)$:
\begin{eqnarray*}
\exp(\ell(\alpha))&\geq& \cosh\left(\ell(\alpha)\right) \\
&\geq & \frac{\cosh\left(\tfrac{1}{2}\ell(\gamma)\right)}{\sinh\left(\tfrac{1}{2}\ell(\beta_1)\right)\sinh\left(\tfrac{1}{2}\ell(\beta_2)\right)}.
\end{eqnarray*}
By comparing the above inequality with the estimates in Case (A)-(C) and using a similar argument,
one can show that $\ell(\alpha)$ is larger than
$$e^t \left(\frac{1}{2}\left(i(\mu,\gamma)-i(\mu,\beta_1)-i(\mu,\beta_2)\right)+\omega_{\beta_1}+\omega_{\beta_2}\right)$$
up to some constant (only depending on $\beta_1,\beta_2$). We omit the details.

\bigskip

We arrive to the following conclusion:

\bigskip

\emph{For any  $\alpha\in \mathcal{A}$, considered as a geodesic arc in $\Gamma(t)$ connecting two simple closed geodesics $\beta_1,\beta_2$,
there are constants $C,T>0$ (depending on $\beta_1$ and $\beta_2$) such that when $t>T$,
$\ell_\alpha(\Gamma(t))\geq e^ti(\mu,\alpha)-C.$}

\begin{lemma}\label{lem:lower}
There is a constant $C>0$ depending only on the stretch line such that for all $\alpha\in \mathcal{A}$,
$$e^ti(\mu,\alpha)-C \leq \ell_\alpha(\Gamma(t)).$$
\end{lemma}

\begin{proof}
Since there are finitely many choices of the pair $\beta_1,\beta_2$ (note that
here $\beta_1,\beta_2$ are boundary components of $S$), we can choose a uniform constant $C$ such that the above
conclusion holds for all $\ell_\alpha(\Gamma(t))$.
\end{proof}

\begin{remark}
One can apply the above argument to give a upper bound for $\ell_\alpha(\Gamma(t))$. That is, one can show that for each $\alpha\in \mathcal{A}$,
there is a constant $C_\alpha>0$ (depending on $\alpha$) such that
\begin{equation}\label{eq:upperbound}
\ell_\alpha(\Gamma(t))\leq e^ti(\mu,\alpha)+C_\alpha.
\end{equation}
To avoid long calculations, we will adopt an indirect method to certify \eqref{eq:upperbound} in the next section.
\end{remark}

\begin{remark}
Our method is close in spirit to \cite[Expos\'e 6, Appendix D]{FLP}. It can be adapted to the case of a general measured lamination and general stretch line,
by specifying an appropriate definition for the intersection number between a measured lamination  and an arc.
\end{remark}
\subsection{Key inequality}

Let $\Gamma(t)$ be a stretch line in $\mathcal{T}(S^d)$ as we have constructed above.
If we restrict each hyperbolic structure $\Gamma(t)$ to the subsurfaces $S$ and $\overline{S}$, then we have
two families of hyperbolic structures on $\mathcal{T}(S)$ and $\mathcal{T}(\overline{S})$, respectively.
We call them $\Gamma_{U}(t)$ and $\Gamma_L(t)$.

It follows directly from Lemma \ref{lem:lower} that there is a constant $C>0$ such that
for any $\alpha\in \mathcal{A}$,

$$e^t i(\mu, \alpha)-C\leq \ell_{\alpha}(\Gamma_{U}(t)),$$
and
 $$e^t i(\bar{\mu}, \bar{\alpha})-C \leq \ell_{\bar{\alpha}}(\Gamma_{L}(t)).$$

  Note that the above inequalities also hold for any simple closed curve $\alpha\in \mathcal{C}$. In this case,
  we can take $C=0$ (this is a consequence of Lemma \ref{lemma:papa}).

Denote by $\overline{\Gamma}_{U}(t)$ and $\overline{\Gamma}_L(t)$  the mirror images of
 $\Gamma_{U}(t)$ and $\Gamma_L(t)$ respectively. Note that $\overline{\Gamma}_{U}(t)\subset \mathcal{T}(\overline{S})$
 and
 $\overline{\Gamma}_L(t)\subset \mathcal{T}(S)$.

 \begin{lemma}\label{lem:left}
 With the above notation, for any $\alpha\in \mathcal{C}\cup \mathcal{A}$, the following
 inequalities hold:

 \begin{equation}
 \begin{cases}
 e^t i(\mu, \alpha)-C\leq \ell_{\alpha}(\Gamma_{U}(t)) \\
  e^t i(\mu, \alpha)-C\leq \ell_{\bar{\alpha}}(\overline{\Gamma}_{U}(t)) \\
 e^t i(\mu, \alpha)-C\leq \ell_{\bar{\alpha}}(\Gamma_{L}(t)) \\
 e^t i(\mu, \alpha)-C \leq \ell_{\alpha}(\overline{\Gamma}_{L}(t)). \\
 \end{cases}
 \end{equation}
 \end{lemma}
Lemma \ref{lem:left} provides a lower bound of the geodesic length of a simple closed curve
or simple arc on $S$ along the path $\Gamma_{U}(t)$. In the following, we will give an upper bound.

 Consider $\alpha\in \mathcal{C}\cup \mathcal{A}$. Denote by $\alpha^d$ the double of
 $\alpha$. Then $\alpha^d$ is either a simple closed curve or the union of two symmetric simple closed curves on $S^d$. Using Lemma \ref{lemma:papa},
 we have a constant $C_\alpha$ such that

 $$\ell_{\alpha^d}(\Gamma(t))\leq e^t i(\mu^d,\alpha^d)+C_\alpha.$$

Note that the sum of the lengths of the two arcs $\ell_{\alpha}(\Gamma(t))$ and $\ell_{\bar{\alpha}}(\Gamma(t))$
is less than  $\ell_{\alpha^d}(\Gamma(t))$. It follows that

\begin{eqnarray*}
\ell_{\alpha}(\Gamma(t))+ \ell_{\bar{\alpha}}(\Gamma(t)) &\leq&   e^t i(\mu^d,\alpha^d)+C_\alpha \\
&=& 2 e^t i(\mu,\alpha) + C_\alpha.
\end{eqnarray*}
Combining the above inequalities with Lemma \ref{lem:left}, we have

\begin{equation}\label{eq:upper}
\begin{cases}
\ell_{\alpha}(\Gamma_U(t))=\ell_{\alpha}(\Gamma(t))\leq e^ti(\mu,\alpha)+C_\alpha+C \\
\ell_{\bar{\alpha}}(\Gamma_L(t))=\ell_{\bar{\alpha}}(\Gamma(t))\leq e^ti(\mu,\alpha)+C_\alpha+C. \\
\end{cases}
\end{equation}

We summarize the above in the following key lemma,
which is a generalization of \cite[Lemma 4.9]{Papa}.

\begin{lemma}\label{lemma:key}
There exists a path $X_t, t\in [0,+\infty)$ in $\mathcal{T}(S)$  such that
each  $\alpha\in \mathcal{C}\cup \mathcal{A}$ satisfies
$$e^{t} i(\mu,\alpha)-C \leq \ell_\alpha(X_t) \leq e^{t} i(\mu,\alpha)+C_\alpha,$$
where $C\geq 0$ is a uniform contant and $C_\alpha>0$ is a constant depending only on $\alpha$.
When $\alpha\in  \mathcal{C}$, we can take $C=0$.
\end{lemma}
Note that the path $X_t$ converges to the point $\mu$ in Thurston's compactification.

\section{Proof of Theorem \ref{thm:main}}\label{sec:main}
\subsection{Ergodic decomposition of a measured lamination} Before we prove the main theorem, we need a generalization of \cite[Lemma 6.4]{Walsh}.

Recall that each measured lamination $\mu$ on $S$ can be decomposed into a finite union of
components, each of which is either a simple closed geodesic, a simple geodesic arc
or a minimal component (each half-leaf is dense).

A  measured lamination $\mu$ is said to be \emph{uniquely ergodic} if the
transverse measure of $\mu$ is the unique measure on the same support up
to a scalar multiple.

More generally, let $\mu$ be an arbitrary  minimal measured lamination on $S$. There exist finitely many invariant transverse
measures $ \mu_1,\cdots ,\mu_p$ on $\mu$  such that
\begin{enumerate}[$\bullet$]
\item $\mu_i$ is ergodic for each $i$.
\item Any invariant transverse measure $\nu$ on $\mu$  can be written as $\nu =\sum_i a_i\mu_i $ for $a_i\geq 0$.
\end{enumerate}
It follows that any measured lamination $\mu$ has a unique decomposition as
$$\mu=\sum a_j\mu_j, a_j\geq 0$$
where each $\mu_j$ is either a simple closed curve, a simple geodesic arc or a minimal geodesic lamination with an ergodic measure.
Such a decomposition is called the \emph{ergodic decomposition} of $\mu$, se e.g. \cite{LM}. 

The following lemma is proved by Walsh \cite[Lemma 6.4]{Walsh} for
surfaces without boundary. His proof works as well for surfaces with boundary.

\begin{lemma}\label{lem:ratio}
Let $\mu=\sum_j \mu_j$ be the ergodic decomposition of $\mu\in \mathcal{PML}$. Then
$$\sup_{\gamma\in \mathcal{C}\cup \mathcal{A}} \frac{i(\nu,\gamma)}{i(\mu,\gamma)}=\max\{f_j\}$$
if $\nu=\sum_j f_j\mu_j$. If $\nu$ cannot expressed as $\sum_j f_j\mu_j, f_j\geq 0$, then the supremum is $+\infty$.
\end{lemma}

\subsection{$\Phi$ is injective}
We will use our construction in Section \ref{sec:inequality}
and some observations on the fine structure of the measured lamination $\mu$.

\begin{figure}[htbp]

\centering

\includegraphics[width=11cm]{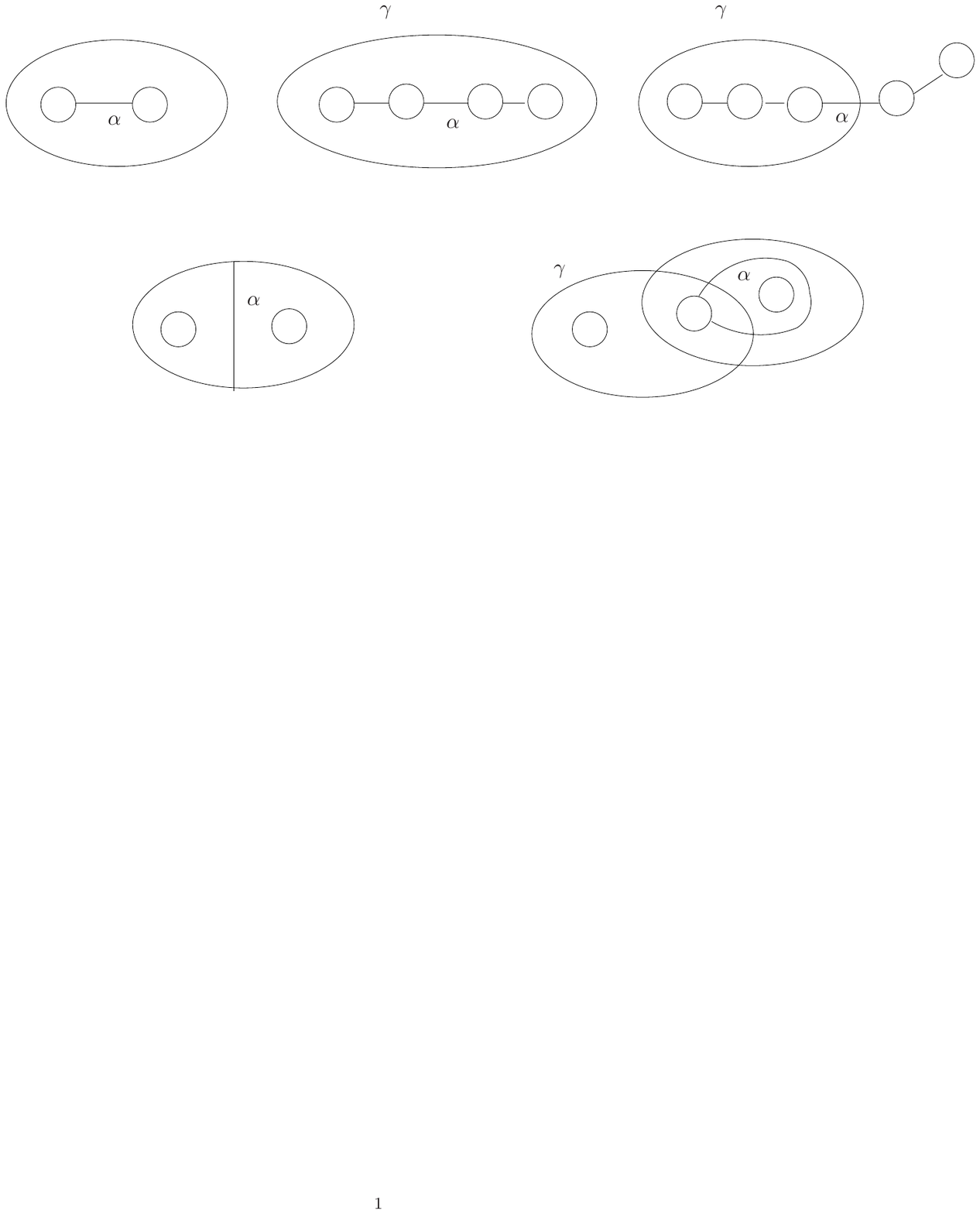}

\caption{The measured lamination $\hat\mu$ is an extension of $\mu$ such that any
geodesic arc on $S$ not contained in $\hat\mu$ is transverse to some simple leaf of $\hat\mu$.}
\label{fig:enlarge}
\end{figure}

Let $\mu=\sum \mu_i$ be the ergodic decomposition of $\mu$. We choose a measured lamination $\hat{\mu}$
which contains $\mu$ as a sublamination, by using the following steps:

\begin{enumerate}[(I)]
\item If $\beta$ is a boundary component of $S$ disjoint from $\mu$, then we add $\beta$
to $\mu$. We get a measured lamination $\mu_0$ such that
$$\mu_0= \mu +\sum_{j=1}^m \beta_j$$
where $\beta_j, j=1,\cdots,m$ are boundary components of $S$ which are disjoint from $\mu$.

\item By definition, $S$ has $p$ ($p\geq m$) boundary components. The numbering is such that for each
$m+1\leq j\leq p$, there is at least an arc contained in $\mu$ that intersects $\beta_j$.
We construct a new measured lamination $\mu_1$ by adding to $\mu_0$ an arc $\alpha_1$ disjoint from $\mu_0$
(if such an arc exists). Inductively, we construct a new measured lamination $\mu_j$ by adding to $\mu_{j-1}$ an arc $\alpha_j$ disjoint from $\mu_{j-1}$. After a finite number of steps, we get a measured lamination $\mu_k$ with the following property:

\begin{quote}
\emph{any arc $\alpha\in \mathcal{A}$ not contained in $\mu_k$ either intersects a simple leaf (an arc or boundary component) of $\mu_k$
or intersect $\mu$.}
\end{quote}

\item By cutting the surface $S$ along all the arcs contained in $\mu_k$, we get a finite union of connected components,
each of which is either a polygon (may be a punctured polygon) or a surface with  piecewise geodesic boundary components.
Let $G$ be a such a component with  piecewise geodesic boundaries.
Let $C$ be a boundary component of $G$. Then $C$ is either a simple closed geodesic contained in $\mu_k$
or a finite concatenations
of geodesic arcs.

In the latter case,
 each geodesic segment of $C$ either comes from an arc in $\mu_k$ (contained as a leaf) or
 a  boundary component of $S$. Note that $C$ is homotopic to a simple closed curve $\gamma$
on $S$, and the geodesic representation of  $\gamma$ is contained in or disjoint from  $\mu_k$.
We will add such a $\gamma$ to $\mu_k$.
The resulting measured lamination, denote by $\hat\mu$ is the one we want.
Note that the choice of $\hat\mu$ is not necessarily unique.
\end{enumerate}

It follows from the above construction that \emph{ for any $\alpha\in \mathcal{A}\cup \mathcal{B}$, either $\alpha$
is a leaf of $\hat\mu$ or $\alpha$ intersects a simple leaf of $\hat\mu$.}

\remark{As we will see later, the refinement (or enlargement) $\hat\mu$ of $\mu$ has better properties when we stretch along 
a complete geodesic lamination on $S^d$ transverse to $\hat\mu^d$, the double of $\hat\mu$. Going forward the stretch line,
any simple arc or boundary curve that intersects $\hat\mu^d$   becomes uniformly large.}

\begin{proposition}\label{pro:injective}
The map \begin{eqnarray*}
\Phi:\overline{\mathcal{T}(S)}&\mapsto& C(\mathcal{T}(S)), \\
 Z&\mapsto& \Phi_Z
\end{eqnarray*} is injective.
\end{proposition}

\begin{proof} We separate the proof into three steps.

\bigskip

\textbf{I}. The restriction $\Phi|_{\mathcal{T}(S)}$ is injective, since for any $X\in \mathcal{T}(S)$,
we have
$$\inf_{Y\in \mathcal{T}(S)}\Phi_X(Y)=-d(X_0,X)$$
and the infimum is exactly obtained at $X$.

\bigskip

\textbf{II}. To show that for any $Y\in \mathcal{T}(S)$ and $\mu\in \mathcal{PML}$,
$\Phi_Y\neq \Phi_\mu$, we observe that $$\inf_{X\in \mathcal{T}(S)}\Phi_{\mu}(X)=-\infty.$$

In fact,  there is a family of hyperbolic structures $X_t$ on $S$ such that $\Phi_\mu (X_t)\to -\infty$. 
As we did in Section \ref{sec:inequality}, we define by
$$\Gamma: t\in \mathbb{R}_+\to \mathcal{T}(S^d)$$
the stretch line converging to the double of $\hat\mu$. We set $X_t=\Gamma^U(t)$.

By Lemma \ref{lemma:key}, there is a constant $C>0$ such that for each $\alpha\in \mathcal{C}\cup \mathcal{A}$,
$$e^ti(\hat\mu,\alpha)\leq \ell_\alpha(X_t)+C.$$
When $\alpha\in \mathcal{C}$, we can take $C$ to be zero.
Let $N>0$ be a sufficiently large constant such that $C/N<1$. Since $\ell_{\hat\mu}(X_t)\to 0$ as $t\to \infty$, by the Collar Lemma,
the length of any geodesic arc intersecting some simple leaf of $\hat\mu$ is be uniformly large
as soon as $t$ is sufficiently large.  By the construction of $\hat\mu$, we have
$$\ell_\alpha(X_t)\geq N, \ \forall \ \alpha\in \mathcal{A}\cup \mathcal{B},i(\hat\mu,\alpha)>0, t\geq T(N).$$

Denote by 
$$C_0= \log\sup_{\eta\in \mathcal{PML}} \frac{i(\mu,\eta)}{\ell_\eta(X_0)}.$$
It follows from the definition of $\Phi_\mu$ that

\begin{eqnarray*}
\Phi_\mu(X_t)&=&\log\sup_{\eta\in \mathcal{PML}} \frac{i(\mu,\eta)}{\ell_\eta(X_t)}- C_0\\
&\leq & \log\sup_{\alpha\in \mathcal{A}\cup \mathcal{B}} \frac{i(\hat\mu,\alpha)}{\ell_\alpha(X_t)} -C_0 \\
&=& \log\sup_{\alpha\in \mathcal{A}\cup \mathcal{B}} \frac{e^t i(\hat\mu,\alpha)}{e^t\ell_\alpha(X_t)} -C_0 \\
&\leq&   \log\sup_{\alpha\in \mathcal{A}\cup \mathcal{B},i(\hat\mu,\alpha)>0} \frac{\ell_\alpha(X_t)+C}{e^t\ell_\alpha(X_t)} -C_0  \\
&\leq& -t + C/N - C_0, \ t> T(N). 
\end{eqnarray*}
As a result, $\Phi_\mu(X_t)\to -\infty$. 

\bigskip

\textbf{III}. It remains to show that for any $\mu\neq \nu$, $\Phi_\mu\neq \Phi_\nu$.
Recall that we made the identification
$$\mathcal{PML}\cong \{\eta\in \mathcal{ML} \ | \ \ell_\eta(X_0)=1\}.$$
Without loss of generality, we assume that
$$ \log\sup_{\eta\in \mathcal{PML}} \frac{i(\mu,\eta)}{\ell_\eta(X_0)}\geq  \log\sup_{\eta\in \mathcal{PML}} \frac{i(\nu,\eta)}{\ell_\eta(X_0)}.$$
Then for any $X\in \mathcal{T}(S)$, we have

\begin{equation}\label{eq:6}
\Phi_\nu(X)-\Phi_\mu(X)\geq \log\sup_{\eta\in \mathcal{PML}} \frac{i(\nu,\eta)}{\ell_\eta(X)}-\log\sup_{\eta\in \mathcal{PML}} \frac{i(\mu,\eta)}{\ell_\eta(X)}.
\end{equation}
We conclude the proof by showing the following lemma, which
 implies that $\Phi_\mu \neq \Phi_\nu$.
\end{proof}

\begin{lemma}\label{lem:good}
There exists a point $Y$ in $\mathcal{T}(S)$ such that
$$\log\sup_{\eta\in \mathcal{PML}} \frac{i(\nu,\eta)}{\ell_\eta(Y)}> \log\sup_{\eta\in \mathcal{PML}} \frac{i(\mu,\eta)}{\ell_\eta(Y)}.$$
\end{lemma}

\begin{proof}[Proof of Lemma \ref{lem:good}]
Let $\hat\mu$ be the refinement of $\mu$ constructed above.  We can write $\hat\mu$ as

$$\hat\mu= \mu+ \zeta.$$
Suppose that $\ell_\zeta(X_0)=L$. Then we set
$$\hat\mu_\epsilon= (1-\epsilon)\mu + \frac{\epsilon}{L}\zeta.$$
It follows that $\hat\mu_\epsilon\in \mathcal{PML}$ for each $0\leq \epsilon\leq 1$.
We first claim that we can find $0<\epsilon<1$ and some $\gamma_0 \in \mathcal{C}\cup \mathcal{A}$ such that

\begin{equation*}
\frac{i(\nu,\gamma_0)}{i(\hat\mu_\epsilon,\gamma_0)}>\frac{1}{1-\epsilon}.
\end{equation*}

We now prove this claim. As above, we assume that  $$\hat\mu_\epsilon=(1-\epsilon)\sum_j \mu_j + \epsilon\sum_k \zeta_k.$$

There are two cases.
If  $\nu$ cannot be expressed as $\sum_j f_j\mu_j+ \sum_k g_k\zeta_k, f_j$ with $g_k\geq 0$,
then (by Lemma \ref{lem:ratio})
$$\sup_{\gamma\in \mathcal{C}\cup \mathcal{A}} \frac{i(\nu,\gamma)}{i(\hat\mu_\epsilon,\gamma)}=\infty.$$
In this case, for any given $0<\epsilon<1$, there is some $\gamma_0 \in \mathcal{C}\cup \mathcal{A}$ such that
$$\frac{i(\nu,\gamma_0)}{i(\hat\mu_\epsilon,\gamma_0)}>\frac{1}{1-\epsilon}.$$

Otherwise,  $$\nu=\sum_j f_j\mu_j+ \sum_k g_k\zeta_k, f_j \ \hbox { for some } \ g_k\geq 0.$$
If there is some $g_k>0$, then we choose $0<\epsilon<1$ sufficiently small such that
$g_j/\epsilon >\frac{1}{1-\epsilon}$. It follows from Lemma \ref{lem:ratio} that
 there is some $\gamma_0 \in \mathcal{C}\cup \mathcal{A}$ such that
$$\frac{i(\nu,\gamma_0)}{i(\hat\mu_\epsilon,\gamma_0)}>\frac{1}{1-\epsilon}.$$

In the case where $\nu=\sum_j f_j\mu_j$, since we assumed that $\ell_\mu(X_0)=\ell_\nu(X_0)=1$, we have
$$\sum_j f_j \ell_{\mu_j}(X_0)= \sum_j \ell_{\mu_j}(X_0)=1.$$
There is some $f_j>1 $. It follows again from Lemma \ref{lem:ratio} that
 there is some $\gamma_0 \in \mathcal{C}\cup \mathcal{A}$ such that
$$\frac{i(\nu,\gamma_0)}{i(\hat\mu_\epsilon,\gamma_0)}\geq \frac{f_j}{1-\epsilon}> \frac{1}{1-\epsilon}.$$

Fix $0< \epsilon< 1$ as above. We assume that

\begin{equation}\label{eq:w}
\frac{i(\nu,\gamma_0)}{i(\hat\mu_\epsilon,\gamma_0)}>\frac{1+\delta}{1-\epsilon}.
\end{equation}
for some sufficiently small constant $\delta>0$.

Like in our proof in Step II,  we denote by
$$\Gamma_\epsilon: t\in \mathbb{R}_+\to \mathcal{T}(S^d)$$
the stretch line converging to the double of $\hat\mu_\epsilon$. We set $X_t=\Gamma_\epsilon^U(t)$.
By the inequality between intersection number and hyperbolic length, we have

\begin{eqnarray*}\sup_{\eta\in \mathcal{PML}}\frac{i(\hat\mu_\epsilon,\eta)}{e^{-t}\ell_\eta(X_t)}&=&\sup_{\alpha\in \mathcal{C}\cup \mathcal{A}}\frac{i(\hat\mu_\epsilon,\eta)}{e^{-t}\ell_\alpha(X_t)} \\
&\leq& \max\{1, \sup_{\alpha\in \mathcal{A},i(\hat\mu_\epsilon,\alpha)>0}\frac{i(\hat\mu_\epsilon,\eta)}{e^{-t}\ell_\alpha(X_t)}\}
\end{eqnarray*}

Let $N>0$ be a sufficiently large constant such that $C/N< \frac{\delta}{3}$. Then 
(by using the Collar Lemma and  the construction of $\hat\mu_\epsilon$ again) we have
$$\ell_\alpha(X_t)\geq N, \ \forall \ \alpha\in \mathcal{A},i(\hat\mu_\epsilon,\alpha)>0, t\geq T(N).$$
It follows that

\begin{eqnarray*}
\sup_{\eta\in \mathcal{PML}}\frac{i(\mu,\eta)}{e^{-t}\ell_\eta(X_t)}&\leq&
\frac{1}{1-\epsilon} \sup_{\eta\in \mathcal{PML}}\frac{i(\hat\mu_\epsilon,\eta)}{e^{-t}\ell_\eta(X_t)} \\
&\leq& \frac{1}{1-\epsilon}\max\{1, \sup_{\alpha\in \mathcal{A},i(\hat\mu_\epsilon,\alpha)>0}\frac{i(\hat\mu_\epsilon,\eta)}{e^{-t}\ell_\alpha(X_t)}\} \\
&\leq& \frac{1}{1-\epsilon}\max\{1, \sup_{\alpha\in \mathcal{A},i(\hat\mu_\epsilon,\alpha)>0}\frac{e^{-t}\ell_\alpha(X_t)+e^{-t}C}{e^{-t}\ell_\alpha(X_t)}\}\\
&\leq& \frac{1+ \frac{C}{N}}{1-\epsilon}\\
&\leq& \frac{1+\delta/3}{1-\epsilon}.
\end{eqnarray*}

As a result, for $t\geq T(N)$ we have

\begin{eqnarray*}
 && \log\sup_{\eta\in \mathcal{PML}} \frac{i(\nu,\eta)}{\ell_\eta(X_t)}-\log\sup_{\eta\in \mathcal{PML}} \frac{i(\mu,\eta)}{\ell_\eta(X_t)} \\
&=& \log\sup_{\eta\in \mathcal{PML}} \frac{i(\nu,\eta)}{e^{-t}\ell_\eta(X_t)}-\log\sup_{\eta\in \mathcal{PML}} \frac{i(\mu,\eta)}{e^{-t}\ell_\eta(X_t)}\\
&\geq& \log \frac{i(\nu,\gamma_0)}{i(\hat\mu_\epsilon, \gamma_0)+e^{-t} C_{\gamma_0}}- \log \frac{1+\frac{\delta}{3}}{1-\epsilon}.
\end{eqnarray*}

By \eqref{eq:w}, when $t$ is sufficiently large, we have
$$\log \frac{i(\nu,\gamma_0)}{i(\hat\mu_\epsilon, \gamma_0)+e^{-t} C_{\gamma_0}}>\log \frac{1+\frac{2\delta}{3}}{1-\epsilon}.$$

This proves Lemma \ref{lem:good}.
\end{proof}

The following result is a consequence of the proof of Proposition \ref{pro:injective}.

\begin{lemma}\label{lem:escape}
Let $(X_n)$ be a sequence in $\mathcal{T}(S)$, $Y\in \mathcal{T}(S)$ such that
$\Phi_{X_n}(\cdot)\to \Phi_Y(\cdot)$. Then
$X_n\to Y$ in $\mathcal{T}(S)$. In particular,
$(X_n)$ cannot escape to infinity in $\mathcal{T}(S)$.
\end{lemma}
\begin{proof}
The lemma is a standard result in a locally compact geodesic metric space,
see Ballmann \cite[Chapter II]{Ballman}. For the Thurston metric on the Teichm\"uller
spaces of surfaces without boundary, see Walsh \cite[Proposition 2]{Walsh} for a proof.

In the case of surfaces with boundary, note that, up to a subsequence, $(X_n)$ converges to a point $P\in \overline{\mathcal{T}(S)}$.
If $P=\mu\in \mathcal{PML}$, then by continuity (see the discussion after Corollary \ref{coro:convergent}), $\Phi_{X_n}\to \Phi_\mu$. By assumption, $\Phi_{X_n}\to \Phi_Y$. Thus $\Phi_{\mu}= \Phi_Y$,
contradiction. If $P\in \mathcal{T}(S)$, it is obvious that $P=Y$. 
\end{proof}

\remark{Let  $(X_n)$ be a sequence in $\mathcal{T}(S)$ converging to $\mu\in \mathcal{PML}$.
Then by the continuity of $\Phi$ ,
we have
$$\Phi_{X_n}(\cdot)\to \Phi_\mu(\cdot)$$
uniformly on any compact subset of $\mathcal{T}(S)$. Note that
$$\inf_{X\in \mathcal{T}(S)}\Phi_{X_n}(X)=-d(X_0,X_n)\to -\infty.$$
However, this does not imply 
$$\inf_{X\in \mathcal{T}(S)}\Phi_{\mu}(X)=-\infty$$ directly, because the infimum may not b attained in $\mathcal{T}(S)$.
It would be interesting to study the level sets of the horofunctions.}

\begin{remark}
The proof of Proposition \ref{pro:injective} applies to the Teichm\"uller space of surfaces without boundary.
This is based again on Lemma \ref{lemma:papa}.
Thus we get a new proof for \cite[Theorem 3.6]{Walsh}. However, the argument in \cite{Walsh} does not
work for surfaces with boundary. Note that in contrast with surfaces without boundary,
 the set of uniquely ergodic measured laminations on  a surface $S$ with boundary is not dense in
$\mathcal{ML}(S)$.
\end{remark}

\begin{theorem}\label{thm:hom}
The map $\Phi$ establishes a homeomorphism between Thurston's compactification $\overline{\mathcal{T}(S)}$ and the horofunction compactification.
\end{theorem}

\begin{proof}

We showed that $\Phi: \overline{\mathcal{T}(S)}\to C(\mathcal{T}(S)) $ is injective and continuous.
Since $\overline{\mathcal{T}(S)}$ is compact, 
any embedding from a compact space to a Hausdorff space is  a homeomorphism
onto its image. As a result, $\Phi(\overline{\mathcal{T}(S)})$ is a compact subset of $C(\mathcal{T}(S))$.
Since the horofuction compactification is the closure of $\Phi(\mathcal{T}(S))$, it must be equal to
 $\Phi(\overline{\mathcal{T}(S)})$.
\end{proof}

\begin{remark}
As we mentioned in the introduction, one of the remaining questions is to understand the isometry group of the arc metric.
One step to handle this question is to calculate the ``detour cost" distance between any two measured laminations on
Thurston's boundary. We will go into details of this calculation in future work.
\end{remark}

Several questions remain open for surfaces with boundary and we mention the following:

\begin{questions}

(a) Is the arc metric Finsler?  If yes, what is the Finsler norm?

(b) Construct families of geodesic between any two points on Teichm\"uller space, analogous to concatenations of stretch lines in the case without boundary.

(c) What is the relation
between the arc metric and the extremal Lipschitz maps between hyperbolic structures?
\end{questions}

Finally, we note that by recent works of Danciger, Gu\'eritaud and Kassel, the
deformation theory of surfaces with boundary is related to Margulis spacetimes in Lorentz geometry \cite{DGK}.
Extremal Lipschitz maps are generalized to geometrically finite hyperbolic manifolds of dimension $n\geq 2$, see \cite{GK}.

\end{document}